%% file: kac-moody-x-centralizer.tex
\pdfoutput=1
\documentclass[10pt]{scrartcl}
\usepackage{amsmath,amsthm,amssymb,mathtools}
\usepackage{nicefrac}
\usepackage{array,tabu,longtable}
\usepackage[longtable]{multirow}
\usepackage{geometry}
\usepackage{calc}

\usepackage[outline]{contour}
\contourlength{0.1em}

\usepackage{tikz}
\usetikzlibrary{arrows,multiplebond,decorations.markings,calc,matrix,shapes.misc,backgrounds,fadings}

\usepackage{listofitems}

\usepackage[colorlinks,allcolors=blue]{hyperref}
\usepackage[capitalize]{cleveref}

\theoremstyle{plain}

\Crefname{theorem}{Theorem}{Theorems}
\numberwithin{equation}{section}

\newtheorem{lemma}{Lemma}
\numberwithin{lemma}{section}
\Crefname{lemma}{Lemma}{Lemmas}

\Crefname{corollary}{Corollary}{Corollaries}

\newtheorem{proposition}[lemma]{Proposition}
\Crefname{proposition}{Proposition}{Propositions}

\newtheorem*{theorem*}{Theorem}
\newtheorem*{lemma*}{Lemma}

\theoremstyle{definition}

\newtheorem{definition}[lemma]{Definition}
\Crefname{definition}{Definition}{Definitions}

\Crefname{example}{Example}{Examples}

\theoremstyle{remark}

\Crefname{remark}{Remark}{Remarks}

\DeclareRobustCommand{\divby}{%
  \mathrel{\vbox{\baselineskip.65ex\lineskiplimit0pt\hbox{.}\hbox{.}\hbox{.}}}%
}

\DeclareMathOperator{\St}{St}
\DeclareMathOperator{\hei}{ht}

\newcommand{\rA}{\mathsf{A}}
\newcommand{\rB}{\mathsf{B}}
\newcommand{\rC}{\mathsf{C}} 
\newcommand{\rD}{\mathsf{D}} 
\newcommand{\rE}{\mathsf{E}}
\newcommand{\rF}{\mathsf{F}}
\newcommand{\rG}{\mathsf{G}}

\tikzstyle{dynkin-vertex} = [draw, circle, thick, minimum size=4mm, inner sep=0.1mm]
\newcommand{\dv}[3]{%
\node[dynkin-vertex] (a#1) at #2 {#3};
}
\tikzfading[name=fade out,inner color=transparent!0,outer color=transparent!100]
\tikzset{dynkinedgelabel/.style={node contents={#1}}}

\newcommand{\geqsign}{\draw (-0.05,0.15) -- (0.1, 0) -- (-0.05, -0.15);}
\newcommand{\leqsign}{\draw (0.05,0.15) -- (-0.1, 0) -- (0.05, -0.15);}

\newcommand{\deA}[3][]{%
\draw[thick,#1] (#2) to (#3);
}
\newcommand{\deAoo}[3][]{%
\draw[thick,double distance=3,postaction=decorate,decoration={{markings,mark=at position 0.3 with {\leqsign},mark=at position 0.7 with{\geqsign}}},#1] (#2) to (#3);
}
\newcommand{\deCtwo}[3][]{%
\draw[thick,double distance=3,postaction=decorate,decoration={markings,mark=at position 0.5 with {\leqsign}},#1] (#2) to (#3);
}
\newcommand{\deGtwo}[3][]{%
\draw[thick,triple,postaction=decorate,decoration={markings,mark=at position 0.5 with {\leqsign}},#1] (#2) to (#3);
}
\newcommand{\deAtt}[3][]{%
\draw[thick,quadruple,postaction=decorate,decoration={markings,mark=at position 0.5 with {\leqsign}},#1] (#2) to (#3);
}
\newcommand{\deH}[5][]{%
\draw[thick,postaction=decorate,decoration={markings,mark=at position 0.25 with {\node[node contents={\contour{white}{#4}}];},mark=at position 0.75 with {\node[node contents={\contour{white}{#5}}];}},#1] (#2) to (#3);
}

\newcommand{\aenroot}[2]{%
\begin{tikzpicture}[baseline=(r-2-1.base)]
\matrix (r) [matrix of math nodes,nodes={draw,rounded rectangle},column sep={.5cm,between borders},ampersand replacement=\&] { #1 };
\draw #2;
\end{tikzpicture}
}

\newcommand{\rootAEnTable}[1]{%
\setlength{\tabcolsep}{4pt}
\begin{tabular}{*{\cslen}{c}} \multirow{2}{*}{\cs[-1]} & \multirow{2}{*}{\cs[1]} & #1 \end{tabular}
}
\newcommand{\rootAEn}[1]{%
\readlist\cs{#1}
\ifcase\cslen\relax
\or rank 1
\or rank 2
\or rank 3
\or \rootAEnTable{\cs[2] \\ && \cs[3]}
\or \rootAEnTable{\cs[2] & \multirow{2}{*}{\cs[3]} \\ && \cs[4]}
\or \rootAEnTable{\cs[2] & \cs[3] \\ && \cs[5] & \cs[4]}
\or \rootAEnTable{\cs[2] & \cs[3] & \multirow{2}{*}{\cs[4]} \\ && \cs[6] & \cs[5]}
\or \rootAEnTable{\cs[2] & \cs[3] & \cs[4] \\ && \cs[7] & \cs[6] & \cs[5]}              
\else unsupported rank
\fi
}

\newcommand{\rootK}[1]{%
\readlist\cs{#1}
{\renewcommand{\arraystretch}{1}
\begin{array}{c@{\ }c@{\ }c}
\multirow{2}{*}{\cs[1]} & \cs[2] & \multirow{2}{*}{\cs[3]} \\ & \cs[4]
\end{array}
}
}

\title{The centralizers of root subgroups\\in Kac—Moody Steinberg groups}
\author{Andrei Smolensky\thanks{\,\ email: \texttt{andrei.smolensky@gmail.com}\newline Department of Mathematics and Mechanics, Saint Petersburg State University\newline The research was supported by RSF (project No. 17-11-01261)}}

\begin{document}

\maketitle

\begin{abstract}
For the affine and hyperbolic root system the symmetric part of the centralizers of root subgroups in the corresponding Steinberg groups are calculated. In the affine case the corresponding root subsystems can be computed in term of the centralizers in the spherical root systems, while in the hyperbolic case there emerges a ``zoo'' of examples, many of them non-hyperbolic. This also delivers many examples of naturally occuring root subsystems of infinite rank.
\end{abstract}

\section{Introduction}

Let $A$ be a symmetrizable GCM, $\Phi$ the associated root system, $\Pi$ its fundamental roots, $\Phi_\mathrm{re}$ the set of real roots.
\begin{definition}
The Kac---Moody Steinberg group $\St(\Phi, R)$ over a commutative ring $R$ is defined by generators $x_\alpha(\xi)$, $\alpha\in\Phi_\mathrm{re}$, $\xi\in R$, and relations
\begin{align}
& x_\alpha(\xi)x_\alpha(\zeta) = x_\alpha(\xi+\zeta), \\
& [x_\alpha(\xi), x_\beta(\zeta)] = \prod_{\qquad\mathclap{i\alpha+j\beta\in C(\alpha,\beta)}} x_{i\alpha+j\beta}(N_{\alpha\beta ij}\xi^i\zeta^j),\quad \text{if the product on the RHS is finite,} \label{eq:comm-general} \\
& w_\alpha(\varepsilon)x_\beta(\xi)w_\alpha(-\varepsilon) = x_{s_\alpha\beta}(\eta_{\alpha\beta}\varepsilon^{-\beta(h_\alpha)}\xi), \quad \varepsilon\in R^*, \label{eq:weyl-conj}
\end{align}
where $C(\alpha, \beta) = (\mathbb{Z}_{>0}\alpha+\mathbb{Z}_{>0}\beta)\cap\Phi_\mathrm{re}$, $N_{\alpha\beta ij}$ are certain integers, $\eta_{\alpha\beta}=\pm1$ and $w_\alpha(\varepsilon)=x_\alpha(\varepsilon)x_{-\alpha}(-\varepsilon^{-1})x_\alpha(\varepsilon)$, while $s_\alpha$ is the simple reflection associated to $\alpha$.
\end{definition}
Assume $\alpha\in\Phi_\mathrm{re}$, we are interested in the structure of $Y=C(x_\alpha(\xi))$, the centralizer of the elementary root unipotent, in $G(\Phi, R)$ or in $\St(\Phi, R)$. An explicit description of $Y$ for finite type root systems was used, for example, by M. Stein in his study of central extensions of Chevalley groups over rings \cite{SteGenRelCov}.

From the definition of $\St(\Phi, R)$ one sees that (in generic case/characteristic $0$) $x_\beta(\zeta)\in Y$ if and only if $C(\alpha, \beta)=\varnothing$, while $h_\beta(\varepsilon)\in Y$ if and only if $\beta(h_\alpha)=0$.

Define
\[ Z(\alpha) = \{ \beta\in\Phi \mid C(\alpha, \beta)=\varnothing \}. \]
This set splits into the union of two parts:
\begin{align*}
& Z_s(\alpha) = Z(\alpha) \cap -Z(\alpha), \\
& Z_u(\alpha) = Z(\alpha) \setminus Z_s(\alpha).
\end{align*}
We are interested in the structure of $Z_s(\alpha)$ and the corresponding subgroup of $\St(\Phi, R)$, which is not the whole (Levi part of the) centralizer, but is reasonably close.

We mostly focus our attention on the case of affine and hyperbolic root systems, because in this case there is a uniform criterion for real roots, which we will use implicitly.
\begin{lemma}[{\cite[Proposition~5.10]{KacInfDimLieAlg}}]
If $A$ is a GCM of finite, affine or hyperbolic type, then
\[ \Phi_\mathrm{re} = \left\{ \alpha=\sum_i k_i\alpha_i \ \middle|\ |\alpha|^2>0,\ k_i\cdot|\alpha_i|^2/|\alpha|^2\in\mathbb{Z} \right\}. \]
\end{lemma}

The paper is organized as follows. In \cref{sec:preliminaries} we recall and establish some basic properties of $Z_s(\alpha)$. In \cref{sec:affine root systems} the centralizers are calculated for the root systems of affine types. The rest of the paper is devoted to the study of indefinite type root systems, with the main focus on the hyperbolic ones. The latter split into three principal cases:
\begin{itemize}
\item root systems of rank $3$ (\cref{sec:rk3}), which is the easiest part;
\item root systems of rank $\geqslant4$ with finite rank centralizers (\cref{sec:rk-geq4}), where some additional computational effort delivers a complete answer;
\item root systems of rank $\geqslant4$ with the centralizers of (apparently) infinite rank (\cref{sec:infinite rank centralizers}), where we only collect some experimental observations.
\end{itemize}
An example of an infinite rank subsystem inside a rank $3$ root system was presented in \cite{MooPiaInfRootSystems} by providing an infinite family of real roots which serve as a ``weak basis'' for a certain subsystem (for a definition of weak independence see \cref{sec:hyperbolic root systems}). The coefficients of these roots are given by polynomials. Here we find many examples of (apparently) infinite-rank subsystems cut out by a simple condition on its elements. Inside it one can find an infinite weakly independent family of roots which is very likely the basis of this subsystem. However, apart from the simplest cases this family is not given by a single polynomial formula. The most striking example is the root system with the following Dynkin diagram:
\[
\begin{tikzpicture}[baseline={(current bounding box.center)}]
\dv{0}{(0,0)}{1}
\dv{1}{(45:1)}{2}
\dv{2}{($(45:1)+(-45:1)$)}{3}
\dv{3}{(-45:1)}{4}
\deA{a0}{a1}\deA{a0}{a2}\deA{a0}{a3}\deA{a1}{a2}\deA{a2}{a3}
\end{tikzpicture}
\]
Here the basis of $Z_s(\alpha_4)$ seems to be covered by infinitely many polynomial families parametrized by a product of two primes congruent to $\pm1$ modulo $12$.

\section{Properties of $Z_s(\alpha)$} \label{sec:preliminaries}
\begin{lemma}
For $\alpha,\beta\in\Phi_\mathrm{re}$ the following statements are equivalent:
\begin{enumerate}
\item $C(\alpha, \beta) = C(\alpha, -\beta) = \varnothing$;
\item $\langle \beta, \alpha^\vee \rangle = 0$ and $\alpha+\beta\notin \Phi_\mathrm{re}$.
\end{enumerate}
\end{lemma}
\begin{proof}
The proof is very similar to the proof of \cite[Lemma~1]{PetKacInfFlagVarConjTheor}.

Consider the $\alpha$-string through $\beta$:
\[ \beta-p\alpha,\ldots,\beta,\ldots,\beta+q\alpha. \]
The first and the last roots of the root string are real, so if $C(\alpha, \beta)=\varnothing$, then $q=0$. Similarly, $p=0$ if $C(\alpha, -\beta) = \varnothing$. But $\langle \beta, \alpha^\vee \rangle = p-q = 0$.

Now suppose that $\langle \beta, \alpha^\vee \rangle = 0$ and $\alpha+\beta\notin\Phi_\mathrm{re}$. Assume that $\gamma = m\alpha+n\beta\in\Phi$ for some $m,n>0$. Note that $w_\alpha w_\beta(\gamma) = -\gamma$, so $\gamma$ is real.

We can choose $\gamma$ so that $m$ is minimal. Consider the $\alpha$-string through $\gamma$:
\[ \gamma-p\alpha,\ldots,\gamma,\ldots,\gamma+q\alpha. \]
Since $m$ is minimal, $\gamma-\alpha\notin\Phi$ unless $m=1$. But then $\gamma-\alpha=n\beta$, which is not a root unless $n=1$, while $\alpha+\beta\notin \Phi_\mathrm{re}$ by assumption. Thus we conclude that $\gamma$ is the first root of this root string and $p=0$. But $p-q=\langle \gamma, \alpha^\vee \rangle = 2m$, so $q<0$, a contradiction.

Since $\alpha+\beta\notin\Phi_\mathrm{re}$ implies $w_\beta(\alpha+\beta)=\alpha-\beta\notin\Phi_\mathrm{re}$ when $\langle \beta, \alpha^\vee \rangle=0$, it follows that $C(\alpha, -\beta)$ is empty as well.
%\todo[inline]{Could we just say that $(\mathbb{Z}\alpha+\mathbb{Z}\beta)\cap\Phi\cong\rC_2$ or $A_1\times A_1$? Apparently not, for how does one ditinguish between the two?}
\end{proof}
As a corollary, one can describe $Z_s(\alpha)_\mathrm{re}$ as
\[ Z_s(\alpha)_\mathrm{re} = \{ \beta\in\Phi_\mathrm{re} \mid \langle\alpha,\beta^\vee\rangle=0 \text{ and } \alpha+\beta\notin\Phi_\mathrm{re} \}. \]

\begin{lemma} \label{lemma:Zs is closed under sum}
$Z_s(\alpha)$ is closed, i.e. if $\beta,\gamma\in Z_s(\alpha)$ and $\beta+\gamma\in\Phi$, then $\beta+\gamma\in Z_s(\alpha)$.
\end{lemma}
\begin{proof}
Let $\beta,\gamma\in Z_s(\alpha)$ be two roots such that $\beta+\gamma\in\Phi$. Suppose that $\beta+\gamma\notin Z_s(\alpha)$. Then, since $\langle \beta+\gamma, \alpha^\vee \rangle = \langle \beta, \alpha^\vee \rangle + \langle \gamma, \alpha^\vee \rangle = 0$, $\alpha+\beta+\gamma$ is a root. Consider the $\beta$-string through $\alpha+\beta+\gamma$:
\[ \alpha+\beta+\gamma,\ldots,\alpha+\beta+\gamma + q\beta. \]
$\alpha+\beta+\gamma$ is the first root of this string because $\alpha+\gamma\notin\Phi$. Then
\[ 0-q = \langle \alpha+\beta+\gamma, \beta^\vee \rangle = 0+2+\langle \gamma, \beta^\vee \rangle, \quad \text{so} \quad \langle \gamma, \beta^\vee \rangle\leqslant-2. \]
Similarly, considering the $\gamma$-string through $\alpha+\beta+\gamma$, one gets $\langle \beta, \gamma^\vee \rangle\leqslant-2$. But in a root system of type $\begin{psmallmatrix}2 & -b \\ -a & 2\end{psmallmatrix}$ with $a,b\geqslant2$ the sum of the two fundamental roots is never a real root, a contradiction.
\end{proof}
\begin{lemma} \label{lemma:Zs is closed under W}
$Z_s(\alpha)$ is closed under the action of the corresponding Weyl group.
\end{lemma}
\begin{proof}
Since $\alpha+\gamma\notin\Phi$, also $s_\beta(\alpha+\gamma)\notin\Phi$. But $\langle \beta,\alpha^\vee \rangle = 0$, so $s_\beta(\alpha+\gamma)=\alpha+s_\beta(\gamma)$, hence $s_\beta(\gamma)\in Z_s(\alpha)$.
\end{proof}
This allows to describe the subgroup of $\St(\Phi,R)$ generated by all $X_\beta$ for $\beta\in Z_s(\alpha)$ as the image of $\St(Z_s(\alpha),R)$ under the homomorphism induced by the inclusion $Z_s(\alpha)\hookrightarrow\Phi$.

\section{Centralizers in affine root systems} \label{sec:affine root systems}
In the affine case the calculation of centralizers can be briefly described as ``affinization of the centralizer in a finite subsystem'' (the centralizers in the finite case are listed in \cite[table after Corollary~5.9]{SteGenRelCov}). Namely, using the standard notation for affine root systems \cite[\S4.8]{KacInfDimLieAlg}, one divides them into three classes as follows: if $\Phi=\Psi^{(r)}$, where $\Psi$ is of finite type and $r=1,2,3$, then the set of real roots is described as follows:

If $r=1$,
\[ \Phi_\mathrm{re} = \{ \alpha+n\delta \mid \alpha\in\Psi,\ n\in\mathbb{Z} \}; \]

If $r=2$ or $3$, $\Phi\neq\rA_{2\ell}^{(2)}$,
\[ \Phi_\mathrm{re} = \{ \alpha+n\delta \mid \alpha\in\Psi^<,\ n\in\mathbb{Z} \} \cup \{ \alpha+nr\delta \mid \alpha\in\Psi^>,\ n\in\mathbb{Z} \}; \]

If $\Phi = \rA_{2\ell}^{(2)}$,
\begin{multline*}
\Phi_\mathrm{re} = \{ \nicefrac{1}{2}\cdot(\alpha+(2n-1)\delta) \mid \alpha\in\Psi^>,\ n\in\mathbb{Z} \} \cup {} \\ \cup \{ \alpha+n\delta \mid \alpha\in\Psi^<,\ n\in\mathbb{Z} \} \cup \{ \alpha+2n\delta \mid \alpha\in\Psi^>,\ n\in\mathbb{Z} \}.
\end{multline*}
Here $\Psi^{<}$ and $\Psi^{>}$ denote the subsets of all short and long roots of $\Psi$ respectively, and $\delta$ is (in the basis of the fundamental roots) the unique entry-wise positive primitive vector in the kernel of the Cartan matrix of $\Phi$ (listed in \cite[\S4.8, Tables Aff~1, Aff~2 and Aff~3]{KacInfDimLieAlg}). In particular, $\delta$ is an isotropic root and is orthogonal to all real roots of $\Phi$.

Consider the case $r=1$. It is easy to see that $\beta+n\delta\in Z_s(\alpha)$ if and only if $\beta\in Z_s(\alpha)$. Indeed, since the expression of a root in $\Psi$ in the form $\beta+n\delta$ is unique, $\alpha+\beta+n\delta\in\Phi_\mathrm{re}$ if and only if $\alpha+\beta\in\Psi$. On the other hand, by the definition of $\delta$ one has $(\beta+n\delta,\alpha) = (\beta,\alpha)$.

A very similar calculation works for other cases. The computation of the centralizers is summarized in \cref{table:affine}.

\begin{table}[h]
\renewcommand{\arraystretch}{1.5}
\centering
\begin{tabu} to 0.9\linewidth {c@{\qquad}c@{\qquad}l@{\qquad}X}
\hline
$\Phi$ & $\alpha$ & $Z_s(\alpha)$ & particular cases \\\hline
$\rA_\ell^{(1)}$ & long & $\rA_{\ell-2}^{(1)}$ & $\varnothing$ for $\ell=2$ \\\hline
$\rB_\ell^{(1)}$ & long & $\rB_{\ell-2}^{(1)}\oplus\rA_1^{(1)}$ & $\rC_2^{(1)}\oplus\rA_1^{(1)}$ for $\ell=4$, $\rA_1^{(1)}\oplus\rA_1^{(1)}$ for $\ell=3$ \\
& short & $\rD_{\ell-1}^{(1)}$ & $\rA_3^{(1)}$ for $\ell=4$, $\rA_1^{(1)}\oplus\rA_1^{(1)}$ for $\ell=3$ \\\hline
$\rC_\ell^{(1)}$ & long & $\rC_{\ell-1}^{(1)}$ & \\
& short & $\rC_{\ell-2}^{(1)}$ & $\rA_1^{(1)}$ for $\ell=3$, $\varnothing$ for $\ell=2$ \\\hline
$\rD_\ell^{(1)}$ & long & $\rD_{\ell-2}^{(1)}\oplus\rA_1^{(1)}$ & $\rA_3^{(1)}\oplus\rA_1^{(1)}$ for $\ell=5$, $3\rA_1^{(1)}$ for $\ell=4$ \\\hline
$\rG_2^{(1)}$ & any & $\rA_1^{(1)}$ & \\\hline
$\rF_4^{(1)}$ & long & $\rC_3^{(1)}$ & \\
& short & $\rA_3^{(1)}$ & \\\hline
$\rE_6^{(1)}$ & long & $\rA_5^{(1)}$ & \\\hline
$\rE_7^{(1)}$ & long & $\rD_6^{(1)}$ & \\\hline
$\rE_8^{(1)}$ & long & $\rE_7^{(1)}$ & \\\hline
$\rA_{2\ell}^{(2)}$ & long & $\rA_{2\ell-2}^{(2)}$ & $\varnothing$ for $\ell=1$ \\
& medium & $\rA_{2\ell-4}\oplus\rA_1^{(1)}$ & $\rA_1^{(1)}$ for $\ell=2$ \\
& short & $\rA_{2\ell-3}^{(2)}$ & \\\hline
$\rA_{2\ell-1}^{(2)}$ & long & $\rA_{2\ell-3}^{(2)}$ & \\
& short & $\rA_{2\ell-5}^{(2)}\oplus\rA_1^{(1)}$ & $2\rA_1^{(1)}$ for $\ell=3$ \\\hline
$\rD_{\ell+1}^{(2)}$ & long & $\rD_{\ell-1}^{(2)}\oplus\rA_1^{(1)}$ & $2\rA_1^{(1)}$ for $\ell=3$, $\rA_1^{(1)}$ for $\ell=2$ \\
& short & $\rB_{\ell-1}^{(1)}$ & $\rC_2^{(1)}$ for $\ell=3$, $\rA_1^{(1)}$ for $\ell=2$ \\\hline
$\rE_6^{(2)}$ & long & $\rA_7^{(2)}$ & \\
& short & $\rB_3^{(1)}$ & \\\hline
$\rD_4^{(3)}$ & any & $\rA_1^{(1)}$ & \\\hline
\end{tabu}
\caption{Centralizers for the root systems of affine types\label{table:affine}}
\end{table}

\section{Centralizers in hyperbolic root systems} \label{sec:hyperbolic root systems}
Following \cite{MooPiaInfRootSystems}, we say that a tuple $\mathcal{D}=(A,\Pi,\Pi^\vee,V,V^\vee,\langle \cdot,\cdot\rangle)$ is a set of root data if $A=(A_{ij})_{i,j\in J}$ is a GCM, $V$ and $V^\vee$ are vector spaces over $\mathbb{R}$, $\langle\cdot,\cdot\rangle\colon V\times V^\vee\to\mathbb{R}$ is a non-degenerate pairing, $\Pi=\{\alpha_i\}_{i\in J}\subset V$ and $\Pi^\vee=\{\alpha_i^\vee\}_{i\in J}\subset V^\vee$ are such that $\langle \alpha_i,\alpha_j^\vee\rangle=A_{ij}$ and $Q=\sum_J\mathbb{Z}\alpha_i$, $Q^\vee=\sum_J\mathbb{Z}\alpha_i^\vee$ are free abelian groups with bases $\{\gamma_i\}_{i\in I}\subset Q$, $\{\gamma_i^\vee\}_{i\in I}\subset Q^\vee$, satisfying
\[ \Pi\subset\bigoplus_I \mathbb{Z}_{\geqslant0}\gamma_i,\quad \Pi^\vee\subset\bigoplus_I \mathbb{Z}_{\geqslant0}\gamma_i^\vee. \]
The fundamental roots $\alpha_i$ are not assumed to be linearly independent, but satisfy the \emph{weak independence property}:
\begin{proposition}[{\cite[Proposition~1]{MooPiaInfRootSystems}}]\label{prop:wip}
For all $k\in J$
\[
\mathbb{Z}\alpha_k \cap \sum_{j\neq k}\mathbb{Z}_{\geqslant0}\alpha_j = \{0\},\quad
\mathbb{Z}\alpha_k^\vee \cap \sum_{j\neq k}\mathbb{Z}_{\geqslant0}\alpha_j^\vee = \{0\}.
\]
\end{proposition}

In \cite[end of Section~7]{MooPiaInfRootSystems} the following procedure for finding a (weak) basis $\Upsilon$ of a subsystem $\Omega$ is indicated:
\begin{enumerate}
\item Fix a well-ordering $\beta_1,\beta_2,\ldots$ on $\Omega_\mathrm{re}^+$, which respects the height.
\item $\beta_1\in\Upsilon$.
\item For $n>1$ set $\beta_n\in\Upsilon$ if and only if $\beta_n\notin\sum_{i=1}^{n-1}\mathbb{Z}_{\geqslant0}\beta_i$.
\end{enumerate}
Checking the weak independence is done by means of integer linear programming. Namely, we use the following variant of integer Farkas' lemma.
\begin{lemma}[Integer Farkas' lemma]
Given $A\in M(m,n,\mathbb{Z})$ and $b\in\mathbb{Z}^m$,
\[ \forall x\in\mathbb{Z}_{\geqslant0}^n\ \ \forall r\in\mathbb{Z}_{>0}\ Ax\neq br \Longleftrightarrow \exists y\in\mathbb{Z}^m\ \ \text{such that}\ \ y^\top A\geqslant0,\ y^\top\cdot b<0. \]
\end{lemma}
Here inequalities of the form ``$\geqslant0$'', ``$<0$'', etc. are applied entry-wise. Note that setting $A$ to be the matrix with columns formed by the coefficient of the first $n$ roots $\gamma_1,\ldots,\gamma_n$ in $\Upsilon$ and $b$ the coefficients column of $\beta$ the left hand side of the equivalence expresses the weak independence of $\beta$ with respect to $\gamma_1,\ldots,\gamma_n$.

Note also that to check the weak independence for a set $\Gamma$ of positive roots it is enough to only check that $m\gamma$, $\gamma\in\Gamma$, $m\neq0$, cannot be expressed as a nonnegative integer linear combination of positive roots $\beta\in\Gamma$ of smaller height than $\gamma$.

So this iterative procedure will reach every element of a basis in finite time. However, \textit{a priori} there is no definite termination point, when one can be sure that all the elements of a basis have been found.

We propose the following procedure to certify the computation of a basis, depending on the rank of $\Phi$. This procedure covers most of the cases, but leaves some for additional manual calculation. Each such case which can be worked out is elaborated below. The notation is as follows: we display the Dynkin diagram of $\Phi$, the matrix of the bilinear form $B$ in the basis of the fundamental roots (which is the symmetrization of the Cartan matrix of $\Phi$), and the root $\alpha$ (meaning that we calculate $Z_s(\alpha)$).

The calculations for the hyperbolic root systems are summarized in the Appendices. The order of appearance of the root systems coincides with the order in the classification in \cite{CarboneEtAlHyperbolicClassification}. The numbering of the root systems is different, for we are only interested in the $142$ symmetrizable systems. The numbering of the fundamental roots is mostly the same, with a few exceptions (those dealt in \cref{sec:infinite rank centralizers}). The particular choice of $\alpha$ is such that the number of the fundamental roots of $\Phi$ lying in $Z_s(\alpha)$ is maximized and also such that the symmetry of the diagram is most preserved. We choose only one fundamental root in each Weyl group orbit, and we also unite roots mapped to each other by a diagram automorphism.

\subsection{Root systems of rank $3$} \label{sec:rk3}
\begin{lemma}
There is no rank $2$ root system in dimension $1$.
\end{lemma}
\begin{proof}
Suppose that the rank $2$ root system with the fundamental roots $\alpha$ and $\beta$ and GCM $\begin{psmallmatrix}
2 & -b \\ -a & 2
\end{psmallmatrix}$ is realized in dimension $1$. That is, $V=\langle v\rangle$ and $V^\vee=\langle v^\vee\rangle$, so that for some non-zero real numbers $\lambda, \lambda^\vee, \mu, \mu^\vee$ one has
\[ \alpha = \lambda v,\quad \alpha^\vee = \lambda^\vee v^\vee,\quad \beta = \mu v,\quad \beta^\vee = \mu^\vee v^\vee. \]
Then one can calculate the values of the pairing $\langle \cdot,\cdot\rangle$ on the pairs of roots:
\begin{align*}
& 2 = \langle \alpha, \alpha^\vee \rangle = \lambda\lambda^\vee\langle v,v^\vee\rangle,
&& -b = \langle \alpha,\beta^\vee \rangle = \lambda\mu^\vee\langle v,v^\vee\rangle, \\
& 2 = \langle \beta, \beta^\vee \rangle = \mu\mu^\vee\langle v,v^\vee\rangle,
&& -a = \langle \beta,\alpha^\vee \rangle = \mu\lambda^\vee\langle v,v^\vee\rangle.
\end{align*}
Denote $\nu=\langle v,v^\vee\rangle$. Note that $\nu\neq0$ by the non-degeneracy of the pairing. Then
\[ \lambda^\vee = \frac{2}{\lambda\nu}, \qquad \mu^\vee = \frac{2}{\mu\nu}, \qquad \frac{-b}{\nu} = \lambda\mu^\vee = \frac{2\lambda}{\mu\nu}, \qquad \frac{-a}{\nu} = \mu\lambda^\vee = \frac{2\mu}{\lambda\nu}. \]
Hence $2\lambda=-b\mu$ and $2\mu=-a\lambda$, and so $2\lambda = \frac{ab\lambda}{2}$, thus $ab=4$.

If $(a,b)=(2,2)$, then $\lambda=-\mu$. If $(a,b)=(1,4)$, then $\lambda=-2\mu$. In both cases the weak independence property fails, a contadiction.
\end{proof}
\begin{lemma}
There is no rank $3$ root system in dimension $2$.
\end{lemma}
\begin{proof}
Suppose that the rank $3$ root system $\Phi$ with the fundamental roots $\alpha,\beta,\gamma$ is realized in dimension $2$. Then, in particular, each of the $3$ subsystems
\[ (\mathbb{Z}\alpha + \mathbb{Z}\beta)\cap\Phi, \qquad (\mathbb{Z}\alpha + \mathbb{Z}\gamma)\cap\Phi, \qquad (\mathbb{Z}\beta + \mathbb{Z}\gamma)\cap\Phi \]
is a rank $2$ root system in dimension $2$.

Assume first that all of these subsystems are infinite. Then by the structure of affine and hyperbolic rank $2$ root systems~\cite{CarboneEtAlRank2Hyperbolic} each one posesses a non-empty cone of imaginary roots, lying between the two fundamental roots. This cones form one of the configurations shown at \cref{fig:rk3dim2}. Since these root systems share the bilinear form, both options are impossible (in one case, the positive part of the cone is disconnected, in the other one of the fundamental roots is not real).
\begin{figure}[h]
\centering
\includegraphics[]{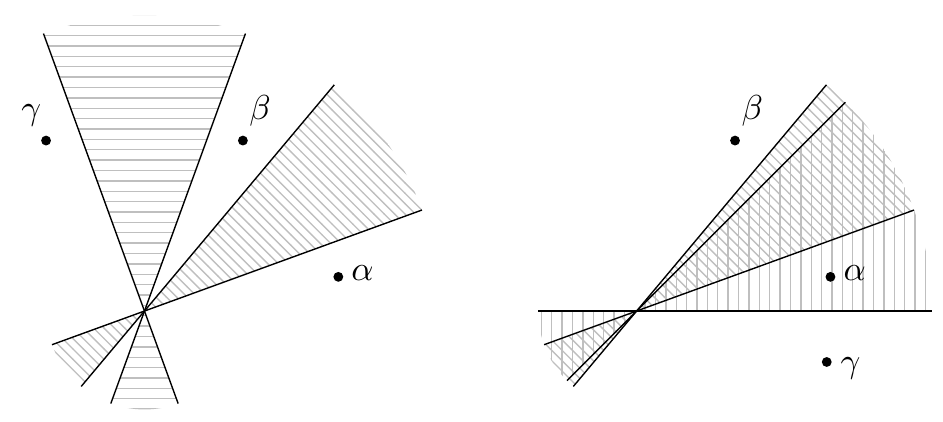}
\caption{Possible planar configurations of $\alpha$, $\beta$ and $\gamma$ (up to reordering) and the imaginary cones of the corresponding rank $2$ subsystems.\label{fig:rk3dim2}}
\end{figure}

If one of these subsystems if finite, then the bilinear form on the ambient $2$-dimensional space is positive-definite, thus $\Phi$ is finite, but in dimension $2$ there are only five non-empty finite root systems ($\rA_1$, $\rA_1\times\rA_1$, $\rA_2$, $\rC_2$, $\rG_2$), all of rank $\leqslant2$.
\end{proof}
As a corollary, once one finds two weakly independent roots in $Z_s(\alpha)$, one can stop, since this subsystem is realized in the $2$-dimensional space $\langle \alpha^\vee \rangle^\perp$. In fact, it is enough to find \emph{any} two roots, for otherwise one gets an $\rA_1$ root system or no roots at all.

All centralizers in hyperbolic root systems of rank $3$ are listed in \cref{app:rank3}. Below we discuss how to certify the basis of $Z_s(\alpha)$ when it is either $\rA_1$ or empty.

\subsubsection{Centralizers of rank $<2$}
Let us now consider the case when there are no two distinct positive roots $\beta\in Z_s(\alpha)$.

Denote $\beta=x_1\alpha_1+x_2\alpha_2+x_3\alpha_3$. Since such $\beta$ is orthogonal to $\alpha$, the coefficients $x_i$ must satisfy the equation of the form
\[ c_1x_1+c_2x_2+c_3x_3 = 0 \]
for some integers $c_i$ (they form a column equal to $B\alpha$). Thus one can eliminate one of the variables, say, $x_1$, from the ternary quadratic form $(\beta,\beta)\in\mathbb{Z}[x_1,x_2,x_3]$ to obtain a binary quadratic form $f$ in variables $x_2,x_3$. In most cases this quadratic form factors into the product of two degree $1$ terms. If so, one can find all non-negative integer solutions to the equations $f(y,z)=|\alpha_i|^2$ for each $i$ (there are only finitely many solutions) and then check whether they deliver a root $\beta$ that is strictly orthogonal to $\alpha$ or, indeed, define a real root at all.

The three remaining cases when no root $\beta\in Z_s(\alpha)$ can be found and the quadratic form $f$ does not factorize are dealt with below.

\paragraph{Case}
\[
\begin{tikzpicture}[baseline={(a1.base)}]
\dv{0}{(120:.75)}{1}\dv{1}{(0:.75)}{2}\dv{2}{(-120:.75)}{3}
\deCtwo{a1}{a0}\deA{a0}{a2}\deCtwo{a1}{a2}
\end{tikzpicture}
\qquad
B = \begin{pmatrix}
2 & -1 & -1 \\ -1 & 1 & -1 \\ -1 & -1 & 2
\end{pmatrix}
\qquad
\alpha = \alpha_2
\qquad
|\beta|^2=1
\]
Compute $0=(\beta,\alpha) = -x+y-z$ and $1=(\beta,\beta) = 2x^2-2xy-2xz+y^2-2yz+2z^2$, the latter simplifies (modulo the former) to $6x^2-6xy+y^2=1$. This implies $y\equiv 1 \pmod{2}$. If $\beta$ is a real root orthogonal to $\alpha$, then $|\beta+\alpha|^2 = |\beta|^2+|\alpha|^2 = 2$, so $\beta+\alpha$ being a real root is equvalent to
\[ 2x\divby2,\quad y+1\divby2,\quad 2z\divby2, \]
which are all satisfied. Hence no root $\beta$ is strictly orthogonal to $\alpha$.

\paragraph{Case}
\[
\begin{tikzpicture}[baseline={(a1.base)}]
\dv{0}{(120:.75)}{1}\dv{1}{(0:.75)}{2}\dv{2}{(-120:.75)}{3}
\deCtwo{a0}{a1}\deA{a0}{a2}\deCtwo{a2}{a1}
\end{tikzpicture}
\qquad
B = \begin{pmatrix}
2 & -2 & -1 \\ -2 & 4 & -2 \\ -1 & -2 & 2
\end{pmatrix}
\qquad
\alpha = \alpha_1
\qquad
|\beta|^2 = 2
\]
Again, such $\beta\in Z_s(\alpha)$ must satisfy $2x-2y-z=0$ and hence $6x^2-24xy+20y^2=2$, which implies $2x^2\equiv2\pmod{4}$, hence $x$ is odd. Now $|\beta+\alpha|^2=4$, and $\beta+\alpha$ is a real root if
\[ 2(x+1)\divby4,\quad 4y\divby4,\quad 2z\divby4, \]
and these conditions are always satisfied.

\paragraph{Case}
\[
\begin{tikzpicture}[baseline={(a1.base)}]
\dv{0}{(120:.75)}{1}\dv{1}{(0:.75)}{2}\dv{2}{(-120:.75)}{3}
\deCtwo{a1}{a0}\deAoo{a0}{a2}\deCtwo{a1}{a2}
\end{tikzpicture}
\qquad
B = \begin{pmatrix}
2 & -1 & -2 \\ -1 & 1 & -1 \\ -2 & -1 & 2
\end{pmatrix}
\qquad
\alpha = \alpha_2
\qquad
|\beta|^2 = 1
\]
$(\beta,\alpha) = -x+y-z$, so $(\beta,\beta) = 8x^2-8xy+y^2=1$, thus $y$ is odd. Since $|\beta+\alpha|^2=2$ and
\[ 2x\divby2,\quad y+1\divby2,\quad 2z\divby2, \]
$\beta+\alpha$ is a real root.

\subsection{Root systems of rank $\geqslant4$} \label{sec:rk-geq4}
For $v\in V^\vee\setminus\{0\}$ define the hyperplane and two half-spaces
\[
H_v=v^\perp=\{x\in V\mid \langle x, v\rangle=0\},\quad
H_v^\pm=\{x\in V\mid \langle x, v\rangle\in\pm\mathbb{R}_{>0}\}
\]
Define the fundamental chamber for $\Phi$ as
\[ F = \{x\in V\mid \langle x, \alpha_j^\vee\rangle >0 \text{ for all } j\in J\}. \]
For a set $\mathcal{H}$ of hyperplanes define the equivalence relation $\sim_\mathcal{H}$ as
\[ x\sim_\mathcal{H}y\quad \Leftrightarrow \quad \text{for every $H\in\mathcal{H}$ either $x,y\in H$ or $x,y\in H^+$ or $x,y\in H^-$}. \]
For a subsystem $\Omega$ of $\Phi$ set $\mathcal{H}(\Omega) = \{ H_{\alpha^\vee} \mid \alpha\in\Omega \}$. Define the fundamental chamber for $\Omega$ as
\[ F(\Omega) = \{x\in V\mid \langle x, \alpha^\vee\rangle >0 \text{ for all } \alpha\in\Omega^+\}. \]
The equivalence relation $\sim_{\mathcal{H}(\Omega)}$ is coarser than $\sim_{\mathcal{H}(\Phi)}$, and hence $F(\Omega)\supset F(\Phi)$. Now $F(\Phi)$ can be defined in term of its fundamental roots as
\[ F(\Phi) = \{x\in V\mid \langle x, \alpha_j^\vee\rangle >0 \text{ for all } j\in J\}. \]
Suppose that $\Gamma\subset\Omega^+$ is a set of roots. They determine a convex cone
\[ C(\Gamma) = \{x\in V\mid \langle x, \alpha^\vee\rangle >0 \text{ for all } \alpha\in\Gamma\}. \]
If $\Gamma$ is a basis for $\Omega$, then $C(\Gamma)=F(\Omega)$, i.e. there is no $\beta\in\Omega$ such that $H_{\beta^\vee}$ cuts $C(\Gamma)$.

Given the chain of closed convex polyhedral cones inclusions $\overline{F(\Phi)}\subseteq \overline{F(\Omega)}\subseteq \overline{C(\Gamma)}$, one can represent each of $F=\overline{F(\Phi)}$ and $C=\overline{C(\Gamma)}$ as the conical hull of a set of vectors:
\[
F = \mathbb{R}_{\geqslant0}r_1+\ldots+\mathbb{R}_{\geqslant0}r_n,\quad
C = \mathbb{R}_{\geqslant0}s_1+\ldots+\mathbb{R}_{\geqslant0}s_k.
\]
With the bounding hyperplanes of the convex polyhedral cone known, the spanning vectors $r_i$ can be calculated via the Double Description method \cite{MotzkinDoubleDescription,FukudaDoubleDescription}.

Now for each face of $F$ one can determine its supporting hyperplane $H$ and find $\alpha\in\Phi$ such that $H=H_{\alpha^\vee}$. Then, if $\alpha\notin\Omega$, the reflection $\sigma$ of $F$ with respect to $H$ leaves the images $\sigma(r_1),\ldots,\sigma(r_n)$ inside $\overline{F(\Omega)}$. One can repeat this procedure for each of the cones $\sigma(F)$ and for each of its faces. If one reaches $s_i$ without ever crossing hyperplanes from $\mathcal{H}(\Omega)$, then one proves that $s_i\in\overline{F(\Omega)}$. If $s_i\in\overline{C(\Gamma)}$ for all $i$, then $C\subseteq\overline{F(\Omega)}$, and $\Gamma$ is the basis of $\Omega$.

However, this procedure never works as stated for $\Omega=Z_s(\alpha)$ and infinite $\Phi$, because in this case $\overline{F(\Omega)}$ contains the line $\mathbb{R}\alpha$, while the Tits cone $\cup_{w\in W}w\overline{F(\Phi)}$ is pointed.

Note that for any $\beta\in\Omega$ one has $\alpha\in H_{\beta^\vee}$, so the hyperplanes from $\mathcal{H}(\Omega)$ cutting $C$ can already be seen on $H_{\alpha^\vee}$. Thus one can look at projections $s_i'$ and $r_i'$ of $s_i$ and $r_i$ onto $H_{\alpha^\vee}$, as well as the projections of all $\sigma(r_i)$, see \cref{fig:cones-proj}.

\begin{figure}[h]
\centering
\includegraphics[]{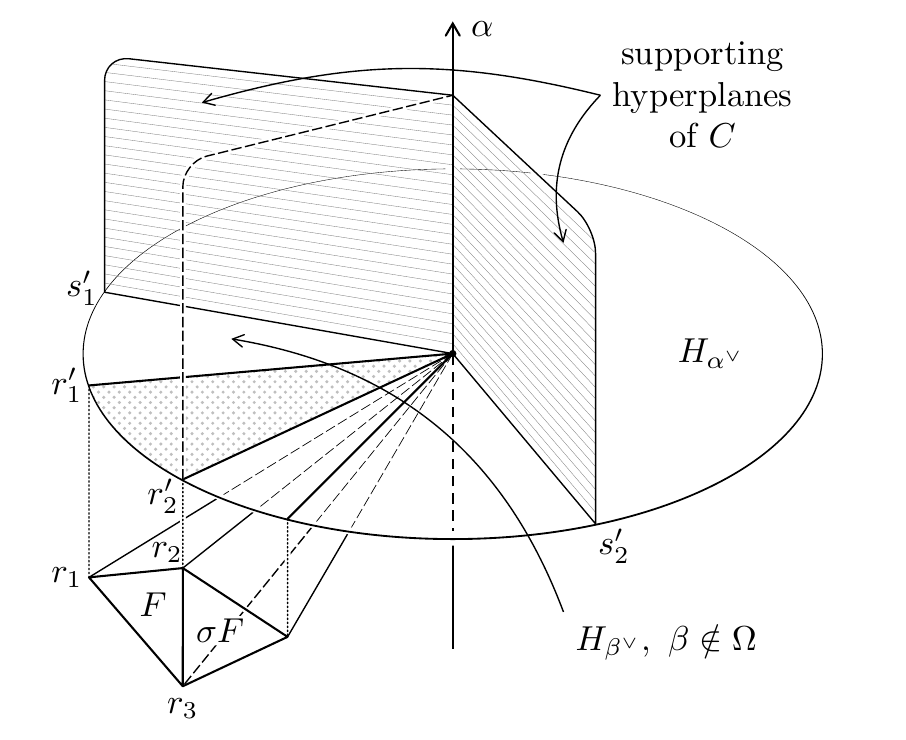}
\caption{The convex cone $F$, spanned by three rays $r_1$, $r_2$, $r_3$, and its reflection $\sigma F$ with respect to $H_{\beta^\vee}$. Once all of $s_i'$ are in the expanding list of $r_i'$, it is proved that $\Gamma$ is a basis for $\Omega$.\label{fig:cones-proj}}
\end{figure}

The bases and the diagrams for all symmetrizable root systems of rank $\geqslant4$ are listed in \cref{app:rank-geq4}, with the exceptions discussed in \cref{sec:infinite rank centralizers}.

The procedure described above allows to certify the basis for most hyperbolic root systems, but fails to work in two extreme cases: no roots in $\Omega$ or too many roots in $\Upsilon$.

\subsubsection{Empty centralizers}
\paragraph{Case}
\[
\begin{tikzpicture}[baseline={($(current bounding box.center)+(0,-4pt)$)}]
\dv{1}{(0,1)}{1}\dv{2}{(1,1)}{2}\dv{3}{(1,0)}{3}\dv{4}{(0,0)}{4}
\deA{a1}{a2}\deA{a1}{a3}\deA{a1}{a4}\deA{a2}{a3}\deA{a2}{a4}\deA{a3}{a4}
\end{tikzpicture}
\qquad
B = \begin{pmatrix}
2&-1&-1&-1\\-1&2&-1&-1\\-1&-1&2&-1\\-1&-1&-1&2
\end{pmatrix}
\qquad
\alpha=\alpha_1
\]
The root $\beta=x\alpha_1+y\alpha_2+z\alpha_3+w\alpha_4$ is orthogonal to $\alpha_1$ if $2x-y-z-w=0$. Substituting the expression for $w$ into the quadratic form $(\beta,\beta)$ one gets that
\[ 2 = |\beta|^2 = 6(x^2-2xy-2xz+y^2+yz+z^2), \]
which is impossible for integer $x,y,z$, thus $Z_s(\alpha_1) = \varnothing$.

The above argument works for any complete graph on $\geqslant2$ points (including $\rA_2$ and $\rA_2^{(1)}$).

\paragraph{Case}
\[
\begin{tikzpicture}[baseline={(a4.base)}]
\dv{1}{(120:1)}{1}\dv{2}{(-120:1)}{2}\dv{3}{(1,0)}{3}\dv{4}{(0,0)}{4}
\deCtwo{a4}{a1}\deCtwo{a4}{a2}\deCtwo{a4}{a3}
\end{tikzpicture}
\qquad
B = \begin{pmatrix}
2 & 0 & 0 & -2 \\ 0 & 2 & 0 & -2 \\ 0 & 0 & 2 & -2 \\ -2 & -2 & -2 & 4
\end{pmatrix}
\qquad
\alpha=\alpha_4
\]
The root $\beta=x\alpha_1+y\alpha_2+z\alpha_3+w\alpha_4$ is orthogonal to $\alpha_4$ if $2w=x+y+z$. Then
\[ 2 = |\beta|^2 = x(x-2y-2z)+(y-z)^2. \]
This equation has no integer solutions, because it has no solutions modulo $4$. Hence $Z_s(\alpha_4) = \varnothing$.

\subsubsection{Infinite rank centralizers} \label{sec:infinite rank centralizers}
In case the Dynkin diagram of $\Phi$ is of the form ``simply-laced cycle + something'', the procedure continues finding new fundamental roots for $Z_s(\alpha)$ indefinitely. This phenomena also occurs for root systems with multiply-laced Dynkin diagrams obtained from those described above by the operation of ``diagram folding''. Below we describe several cases when $Z_s(\alpha)$ has infinite rank.

\paragraph{Case}
%\begin{center}
\[
\begin{tikzpicture}[baseline={(current bounding box.center)}]
\dv{0}{(0,0)}{1}
\dv{1}{(30:1)}{2}
\dv{2}{(-30:1)}{3}
\dv{3}{(-1,0)}{4}
\deA{a0}{a1}\deA{a0}{a2}\deA{a0}{a3}\deA{a1}{a2}
\end{tikzpicture}
\]
%\end{center}
The basis for $Z_s(\alpha_4)$ is provided by the family of roots
%\[ \beta_n = \begin{pmatrix}
%6n(n+1) \\ n(3n+2) \\ (n+1)(3n+1) \\ 3n(n+1)
%\end{pmatrix},\quad n\in\mathbb{Z}_{\geqslant0} \]
\[ \beta_n = \aenroot{\&\& n(3n+2) \\ 3n(n+1) \& 6n(n+1) \\ \&\& (n+1)(3n+1) \\}{(r-2-1)--(r-2-2)--(r-1-3)--(r-3-3)--(r-2-2)},\quad n\in\mathbb{Z}_{\geqslant0} \]
and those obtained from these by swapping the coefficients of $\alpha_2$ and $\alpha_3$ (we call them $\beta'_n$). In particular, $\beta_0=\alpha_3$ and $\beta'_0=\alpha_2$. Set
\begin{gather*}
y = (-n,n,n-1,1),\qquad y' = (-n,n-1,n,1),\\
\intertext{so that for $m<n$}
y\cdot\beta_m = n-m-1 \geqslant 0, \qquad y\cdot\beta'_m = m+n \geqslant 0, \qquad y\cdot\beta_n = -1 < 0, \\
y'\cdot\beta_m = m+n \geqslant 0, \qquad y'\cdot\beta'_m = n-m-1 \geqslant 0, \qquad y'\cdot\beta'_m = -1 < 0.
\end{gather*}
Hence by Integer Farkas lemma $\beta_n$ and $\beta'_n$ cannot be expressed in terms of $\beta_m,\beta'_m$, $m<n$. Since $\hei(\beta_n)=\hei(\beta'_n)$, we see that $\beta_1,\beta'_1,\ldots,\beta_n,\beta'_n$ are weakly independent.

%To show that $\Upsilon = \{ \beta_n,\beta'_n \mid n\in\mathbb{Z}_{\geqslant0} \}$ form a basis, we emulate the procedure described in \cref{sec:rk-geq4}. Namely, set
%\[ \Upsilon_n = \{ \beta_k,\beta'_k \mid k=0,\ldots,n \},\quad C_n = \operatorname{pr}_{H_{\alpha_4^\vee}}\left(\overline{C(\Upsilon_n)}\right),\quad n\in\mathbb{Z}_{>0}. \]
%Then the rays of $C_n$ are
%\[ -\begin{pmatrix}
%12k^2+1 \\ 6k^2\mp2k+2 \\ 6k^2\pm2k+2 \\ 6k^2-1
%\end{pmatrix},\quad k=0,\ldots,n,\qquad\text{and}\qquad
%-\begin{pmatrix}
%12n^2+12n-2 \\ 6n^2+6n \\ 6n^2+6n \\ 6n^2+6n-1
%\end{pmatrix}. \]

%Indeed, $\beta=x\alpha_1+y\alpha_2+z\alpha_3+w\alpha_4$ is orthogonal to $\alpha_1$ if $2x=y$, thus
%\[ \frac{1}{2}(\beta,\beta) = 3x^2 - (2x+z)w -2xz + w^2 + z^2 = 1. \]
%Solving this equation modulo $3$ and $4$ shows that $x\divby3$ and $x\divby2$.

\paragraph{Case}
\[
\begin{tikzpicture}[baseline={(current bounding box.center)}]
\dv{4}{(-1,0)}{5}
\dv{0}{(0,0)}{1}
\dv{1}{(45:1)}{2}
\dv{2}{($(45:1)+(-45:1)$)}{3}
\dv{3}{(-45:1)}{4}
\deA{a0}{a1}
\deA{a1}{a2}
\deA{a2}{a3}
\deA{a3}{a0}
\deA{a4}{a0}
\end{tikzpicture}
\]
The basis for $Z_s(\alpha_5)$ is given by $\alpha_2,\alpha_3,\alpha_4,2\alpha_1+\alpha_2+\alpha_4+\alpha_5$ together with the family of roots
\[
\aenroot{\& \& n(2n+1) \\ 2n(n+1) \& 4n(n+1) \& \& 2n(n+1) \\ \& \& (n+1)(2n+1) \\}{(r-2-1)--(r-2-2)--(r-1-3)--(r-2-4)--(r-3-3)--(r-2-2)},\quad n\in\mathbb{Z}_{\geqslant1}
\]
and those obtained frome these by swapping the coefficients for $\alpha_2$ and $\alpha_4$.

\paragraph{Case}
\[
\begin{tikzpicture}[baseline={(current bounding box.center)}]
\dv{5}{(-1,0)}{6}
\dv{0}{(0,0)}{1}
\dv{1}{(30:1)}{2}
\dv{2}{($(30:1)+(0:1)$)}{3}
\dv{3}{($(-30:1)+(0:1)$)}{4}
\dv{4}{(-30:1)}{5}
\deA{a0}{a1}
\deA{a1}{a2}
\deA{a2}{a3}
\deA{a3}{a4}
\deA{a4}{a0}
\deA{a5}{a0}
\end{tikzpicture}
\]
In this case experimental evidence suggests that the basis of $Z_s(\alpha_6)$ is not given by polynomials. It starts with $\alpha_2,\ldots,\alpha_5$ together with the following roots (and their symmetric images):

\begin{center}
\begin{tabular}{cccc}
$\rootAEn{2,1,0,0,1,1}$,
& $\rootAEn{6,2,3,4,5,3}$,
& $\rootAEn{20,8,9,11,13,10}$, \\[15pt]
$\rootAEn{40,17,19,21,24,20}$,
& $\rootAEn{66,29,32,35,38,33}$,
& $\rootAEn{100,45,48,52,56,50}$, \\[15pt]
$\rootAEn{140,64,68,72,77,70}$,
& $\rootAEn{186,86,91,96,101,93}$,
& $\rootAEn{240,112,117,123,129,120}$, \\[15pt]
$\rootAEn{300,141,147,153,160,150}$,
& $\rootAEn{368,174,180,187,195,184}$,
& $\ldots$
\end{tabular}
\end{center}

\paragraph{Case}
\[
\begin{tikzpicture}[baseline={(current bounding box.center)}]
\dv{6}{(-1,0)}{7}
\dv{0}{(0,0)}{1}
\dv{1}{(45:1)}{2}
\dv{2}{($(45:1)+(0:1)$)}{3}
\dv{3}{($(45:1)+(0:1)+(-45:1)$)}{4}
\dv{4}{($(-45:1)+(0:1)$)}{5}
\dv{5}{(-45:1)}{6}
\deA{a0}{a1}
\deA{a1}{a2}
\deA{a2}{a3}
\deA{a3}{a4}
\deA{a4}{a5}
\deA{a5}{a0}
\deA{a6}{a0}
\end{tikzpicture}
\]
The basis of $Z_s(\alpha_7)$ starts with $\alpha_2,\ldots,\alpha_6$ and the following roots (together with their symmetric images):
\begin{center}
\begin{tabular}{cccc}
$\rootAEn{2,1,0,0,0,1,1}$
& $\rootAEn{6,2,2,3,4,5,3}$
& $\rootAEn{18,7,8,9,10,12,9}$, \\[15pt]
$\rootAEn{36,15,16,18,20,22,18}$,
& $\rootAEn{60,26,28,30,32,35,30}$,
& $\rootAEn{90,40,42,45,48,51,45}$, \\[15pt]
$\rootAEn{126,57,60,63,66,70,63}$,
& $\rootAEn{168,77,80,84,88,92,84}$,
& $\rootAEn{216, 100, 104, 108, 112, 117, 108}$ \\[15pt]
$\rootAEn{270, 126, 130, 135, 140, 145, 135}$,
& $\rootAEn{330, 155, 160, 165, 170, 176, 165}$,
& $\ldots$
\end{tabular}
\end{center}
The roots listed above are covered (up to symmetry) by two disjoint polynomial families
\begin{align*}
& \aenroot{\&\& 6n^2-5n+1 \& 6n^2-4n \\ 3n(2n-1) \& 6n(2n-1) \&\&\& 3n(2n-1) \\ \&\& 6n^2-n \& 6n^2-2n \\}{(r-2-1)--(r-2-2)--(r-1-3)--(r-1-4)--(r-2-5)--(r-3-4)--(r-3-3)--(r-2-2)},\quad n\in\mathbb{Z}_{\geqslant1}, \\
& \aenroot{\&\& 6n^2+n \& 6n^2+2n \\ 3n(2n+1) \& 6n(2n+1) \&\&\& 3n(2n+1) \\ \&\& 6n^2+5n+1 \& 6n^2+4n \\}{(r-2-1)--(r-2-2)--(r-1-3)--(r-1-4)--(r-2-5)--(r-3-4)--(r-3-3)--(r-2-2)},\quad n\in\mathbb{Z}_{\geqslant1}.
\end{align*}
Note that they can be combined into a single ``almost polynomial'' family as even and odd members of the sequence:
\[
\aenroot{\& \& \frac{1}{2}n(3n+1) \& \makebox[\widthof{666}][c]{?} \\ \frac{3}{2}n(n+1) \& 3n(n+1) \& \& \& \frac{3}{2}n(n+1) \\ \& \& \frac{1}{2}(3n^2+5n+2) \& \makebox[\widthof{666}][c]{?} \\}{(r-2-1)--(r-2-2)--(r-1-3)--(r-1-4)--(r-2-5)--(r-3-4)--(r-3-3)--(r-2-2)},\quad n\in\mathbb{Z}_{\geqslant1}
\]

\paragraph{Case}
\[
\begin{tikzpicture}[baseline={(current bounding box.center)}]
\dv{0}{(0,0)}{1}
\dv{1}{(45:1)}{2}
\dv{2}{($(45:1)+(-45:1)$)}{3}
\dv{3}{(-45:1)}{4}
\deA{a0}{a1}\deA{a0}{a2}\deA{a0}{a3}\deA{a1}{a2}\deA{a2}{a3}
\end{tikzpicture}
\]
\begin{table}[h]
\setlength{\tabcolsep}{1pt}
\renewcommand{\arraystretch}{3}
{\scriptsize
\begin{tabular}{*{8}{c}}
$\rootK{4,1,8,6}$ & $\rootK{8,13,28,18}$ & $\rootK{28,37,92,60}$ & $\rootK{28,61,104,66}$ & $\rootK{60,73,192,126}$ & $\rootK{52,97,188,120}$ & $\rootK{104,121,328,216}$ & $\ldots$ \\
$\rootK{20,1,28,24}$ & $\rootK{20,13,52,36}$ & $\rootK{48,37,132,90}$ & $\rootK{120,61,288,204}$ & $\rootK{88,73,248,168}$ & $\rootK{168,97,420,294}$ & $\rootK{140,121,400,270}$ \\
$\rootK{48,1,60,54}$ & $\rootK{120,13,192,156}$ & $\rootK{368,37,580,474}$ & $\rootK{500,61,820,660}$ & $\rootK{748,73,1172,960}$ & $\vdots$ & $\vdots$ \\
$\rootK{88,1,104,96}$ & $\rootK{168,13,252,210}$ & $\rootK{448,37,680,564}$ & $\vdots$  & $\vdots$ &  &  \\
$\rootK{140,1,160,150}$ & $\rootK{388,13,512,450}$ & $\vdots$ &  &  &  &  \\
$\rootK{204,1,228,216}$ & $\rootK{472,13,608,540}$ &  &  &  &  &  \\
$\rootK{280,1,308,294}$ & $\rootK{812,13,988,900}$ &  &  &  &  &  \\
$\rootK{368,1,400,384}$ & $\rootK{932,13,1120,1026}$ &  &  &  &  &  \\
$\rootK{468,1,504,486}$ & $\vdots$  &  &  &  &  &  \\
$\rootK{580,1,620,600}$ &  &  &  &  &  &  \\
$\rootK{704,1,748,726}$ &  &  &  &  &  &  \\
$\rootK{840,1,888,864}$ &  &  &  &  &  &  \\
$\rootK{988,1,1040,1014}$ &  &  &  &  &  & \\
$\vdots$
\end{tabular}
}
\caption{The basis of $Z_s(\alpha_4)$ in the root system №45. Each root comes with its symmetric image (the coefficients for $\alpha_1$ and $\alpha_3$ are swapped).\label{table:k2m24}}
\end{table}
The beginning of the basis of $Z_s(\alpha_4)$ is listed in \cref{table:k2m24}. From the experimental evidence it seems that some subsets of the basis of $Z_s(\alpha_4)$ are (up to symmetry swapping $\alpha_1$ and $\alpha_3$) parametrized by polynomials. The roots with $c_2=1$ are of the form
\[
\aenroot{\& \makebox[\widthof{$6n^2$}][c]{1} \\ 2n(3n-1) \& \& 2n(3n+1) \\ \& 6n^2 \\}{(r-2-1)--(r-1-2)--(r-2-3)--(r-3-2)--(r-2-1)--(r-2-3)},\quad n\in\mathbb{Z}_{\geqslant0}.
\]
Roots with $c_2=13$ are covered by
\begin{align*}
& \aenroot{\& \makebox[\widthof{666}][c]{13} \\ 78n^2-122n+52 \& \& 78n^2-70n+20 \\ \& 78n^2-96n+36 \\}{(r-2-1)--(r-1-2)--(r-2-3)--(r-3-2)--(r-2-1)--(r-2-3)},\quad n\in\mathbb{Z}_{\geqslant0}, \\
& \aenroot{\& \makebox[\widthof{666}][c]{13} \\ 78n^2-86n+28 \& \& 78n^2-34n+8 \\ \& 78n^2-60n+18 \\}{(r-2-1)--(r-1-2)--(r-2-3)--(r-3-2)--(r-2-1)--(r-2-3)},\quad n\in\mathbb{Z}_{\geqslant0}
\end{align*}
Roots with $c_2=37$ are covered by
\begin{align*}
& \aenroot{\& \makebox[\widthof{666}][c]{37} \\ 222n^2-326n+132 \& \& 22n^2-178n+48 \\ \& 222n^2-252n+90 \\}{(r-2-1)--(r-1-2)--(r-2-3)--(r-3-2)--(r-2-1)--(r-2-3)},\quad n\in\mathbb{Z}_{\geqslant0}, \\
& \aenroot{\& \makebox[\widthof{666}][c]{37} \\ 222n^2-266n+92 \& \& 222n^2-118n+28 \\ \& 222n^2-192n+60 \\}{(r-2-1)--(r-1-2)--(r-2-3)--(r-3-2)--(r-2-1)--(r-2-3)},\quad n\in\mathbb{Z}_{\geqslant0}
\end{align*}
The found possible values of $c_2$ are
\[ 1,13,37,61,73,97,121,181,193,241,253,277,337,349,397,421,433,481,517,673,781,\ldots \]
The order of their first appearances is not coherent with height. All of them except for $1$ are either primes $p$, $p\equiv1\pmod{12}$, or products of two primes $pq$, $p\equiv q\equiv\pm1\pmod{12}$. It is not clear whether the products with $p=q$ appear in this sequence except fo $p=q=11$.

\bibliographystyle{amsalpha}
\bibliography{kac-moody-x-centralizer}

\appendix
\clearpage
\section{Centralizers in rank $3$} \label{app:rank3}
\input{table-rk-leq3-test}

\clearpage
\newgeometry{margin=1cm}
\section{Centralizers in rank $\geqslant 4$} \label{app:rank-geq4}
%\clearpage
\tikzset{every picture/.append style={scale=0.75}}
\tikzset{dynkin-vertex/.append style={font=\scriptsize}}
%\tiny
%\scriptsize
%\footnotesize
\small
\input{table-rk-geq4}
\restoregeometry
\normalsize

\end{document}

%% file: table-rk-leq3-test.tex
{\tikzset{every picture/.append style={show background rectangle,inner frame sep=1ex,background rectangle/.style={use as bounding box}}}
\renewcommand{\arraystretch}{0}
\noindent
\begin{longtable}{|l|@{}c@{}|@{}c@{}|}
\hline
№ & $\Phi$ & \begin{tabular}{m{1cm}|m{3cm}|m{4cm}} \hfil$\alpha$\hfil & $Z_s(\alpha)$ & basis \end{tabular} \\\hline % & $Z_s(\alpha)$ & $\Pi_\alpha$ \\\hline

1 &
\begin{tikzpicture}[baseline={(a2.base)}]
\dv{1}{(0,0)}{1}\dv{2}{(-30:1)}{2}\dv{3}{(-90:1)}{3}
\deCtwo{a2}{a1}\deA{a1}{a3}\deCtwo{a2}{a3}
\end{tikzpicture}
&\renewcommand{\arraystretch}{1.75}%
\begin{tabular}{m{1cm}|m{3cm}|m{4cm}}
\hfil $\alpha_{1}$ \hfil &
\hfil \begin{tikzpicture}[baseline={(a2.base)}]
\dv{1}{(0, 0)}{1}\dv{2}{(1.5, 0)}{2}
\deH{a1}{a2}{10}{10}
\end{tikzpicture} \hfil
&
(1, 2, 0)\newline (3, 2, 4)
\\\hline
\hfil $\alpha_{2}$ \hfil &
\hfil {\Huge $\varnothing$} \rule[-6pt]{0pt}{28pt} \hfil
&

\end{tabular}\\\hline
2 &
\begin{tikzpicture}[baseline={(a2.base)}]
\dv{1}{(0,0)}{1}\dv{2}{(-30:1)}{2}\dv{3}{(-90:1)}{3}
\deCtwo{a1}{a2}\deA{a1}{a3}\deCtwo{a3}{a2}
\end{tikzpicture}
&\renewcommand{\arraystretch}{1.75}%
\begin{tabular}{m{1cm}|m{3cm}|m{4cm}}
\hfil $\alpha_{1}$ \hfil &
\hfil {\Huge $\varnothing$} \rule[-6pt]{0pt}{28pt} \hfil
&

\\\hline
\hfil $\alpha_{2}$ \hfil &
\hfil \begin{tikzpicture}[baseline={(a2.base)}]
\dv{1}{(0, 0)}{1}\dv{2}{(1.5, 0)}{2}
\deH{a1}{a2}{4}{4}
\end{tikzpicture} \hfil
&
(2, 1, 0)\newline (0, 1, 2)
\end{tabular}\\\hline
3 &
\begin{tikzpicture}[baseline={(a2.base)}]
\dv{1}{(0,0)}{1}\dv{2}{(-30:1)}{2}\dv{3}{(-90:1)}{3}
\deGtwo{a2}{a1}\deA{a1}{a3}\deGtwo{a2}{a3}
\end{tikzpicture}
&\renewcommand{\arraystretch}{1.75}%
\begin{tabular}{m{1cm}|m{3cm}|m{4cm}}
\hfil $\alpha_{1}$ \hfil &
\hfil \begin{tikzpicture}[baseline={(a2.base)}]
\dv{1}{(0, 0)}{1}\dv{2}{(1.5, 0)}{2}
\deH{a1}{a2}{34}{34}
\end{tikzpicture} \hfil
&
(1, 2, 0)\newline (3, 2, 4)
\\\hline
\hfil $\alpha_{2}$ \hfil &
\hfil \begin{tikzpicture}[baseline={(a2.base)}]
\dv{1}{(0, 0)}{1}\dv{2}{(1.5, 0)}{2}
\deH{a1}{a2}{10}{10}
\end{tikzpicture} \hfil
&
(2, 3, 0)\newline (0, 3, 2)
\end{tabular}\\\hline
4 &
\begin{tikzpicture}[baseline={(a2.base)}]
\dv{1}{(0,0)}{1}\dv{2}{(-30:1)}{2}\dv{3}{(-90:1)}{3}
\deGtwo{a1}{a2}\deA{a1}{a3}\deGtwo{a3}{a2}
\end{tikzpicture}
&\renewcommand{\arraystretch}{1.75}%
\begin{tabular}{m{1cm}|m{3cm}|m{4cm}}
\hfil $\alpha_{1}$ \hfil &
\hfil \begin{tikzpicture}[baseline={(a2.base)}]
\dv{1}{(0, 0)}{1}\dv{2}{(1.5, 0)}{2}
\deH{a1}{a2}{34}{34}
\end{tikzpicture} \hfil
&
(3, 2, 0)\newline (9, 2, 12)
\\\hline
\hfil $\alpha_{2}$ \hfil &
\hfil \begin{tikzpicture}[baseline={(a2.base)}]
\dv{1}{(0, 0)}{1}\dv{2}{(1.5, 0)}{2}
\deH{a1}{a2}{10}{10}
\end{tikzpicture} \hfil
&
(2, 1, 0)\newline (0, 1, 2)
\end{tabular}\\\hline
5 &
\begin{tikzpicture}[baseline={(a2.base)}]
\dv{1}{(-1.25, 0)}{1}\dv{2}{(0, 0)}{2}\dv{3}{(1.25, 0)}{3}
\deGtwo{a2}{a1}\deCtwo{a3}{a2}
\end{tikzpicture}
&\renewcommand{\arraystretch}{1.75}%
\begin{tabular}{m{1cm}|m{3cm}|m{4cm}}
\hfil $\alpha_{1}$ \hfil &
\hfil \begin{tikzpicture}[baseline={(a2.base)}]
\dv{1}{(0, 0)}{1}\dv{2}{(1.5, 0)}{2}
\deH{a1}{a2}{2}{4}
\end{tikzpicture} \hfil
&
(0, 0, 1)\newline (1, 2, 0)
\\\hline
\hfil $\alpha_{2}$ \hfil &
\hfil \begin{tikzpicture}[baseline={(a2.base)}]
\dv{1}{(0, 0)}{1}\dv{2}{(1.5, 0)}{2}
\deH{a1}{a2}{2}{6}
\end{tikzpicture} \hfil
&
(0, 1, 2)\newline (2, 3, 0)
\\\hline
\hfil $\alpha_{3}$ \hfil &
\hfil \begin{tikzpicture}[baseline={(a2.base)}]
\dv{1}{(0, 0)}{1}\dv{2}{(1.5, 0)}{2}
\deH{a1}{a2}{4}{4}
\end{tikzpicture} \hfil
&
(1, 0, 0)\newline (1, 6, 6)
\end{tabular}\\\hline
6 &
\begin{tikzpicture}[baseline={(a2.base)}]
\dv{1}{(-1.25, 0)}{1}\dv{2}{(0, 0)}{2}\dv{3}{(1.25, 0)}{3}
\deGtwo{a1}{a2}\deCtwo{a2}{a3}
\end{tikzpicture}
&\renewcommand{\arraystretch}{1.75}%
\begin{tabular}{m{1cm}|m{3cm}|m{4cm}}
\hfil $\alpha_{1}$ \hfil &
\hfil \begin{tikzpicture}[baseline={(a2.base)}]
\dv{1}{(0, 0)}{1}\dv{2}{(1.5, 0)}{2}
\deH{a1}{a2}{4}{2}
\end{tikzpicture} \hfil
&
(0, 0, 1)\newline (3, 2, 0)
\\\hline
\hfil $\alpha_{2}$ \hfil &
\hfil \begin{tikzpicture}[baseline={(a2.base)}]
\dv{1}{(0, 0)}{1}\dv{2}{(1.5, 0)}{2}
\deH{a1}{a2}{10}{10}
\end{tikzpicture} \hfil
&
(2, 1, 0)\newline (2, 3, 2)
\\\hline
\hfil $\alpha_{3}$ \hfil &
\hfil \begin{tikzpicture}[baseline={(a2.base)}]
\dv{1}{(0, 0)}{1}\dv{2}{(1.5, 0)}{2}
\deH{a1}{a2}{1}{6}
\end{tikzpicture} \hfil
&
(1, 0, 0)\newline (0, 2, 1)
\end{tabular}\\\hline
7 &
\begin{tikzpicture}[baseline={(a2.base)}]
\dv{1}{(-1.25, 0)}{1}\dv{2}{(0, 0)}{2}\dv{3}{(1.25, 0)}{3}
\deGtwo{a2}{a1}\deCtwo{a2}{a3}
\end{tikzpicture}
&\renewcommand{\arraystretch}{1.75}%
\begin{tabular}{m{1cm}|m{3cm}|m{4cm}}
\hfil $\alpha_{1}$ \hfil &
\hfil \begin{tikzpicture}[baseline={(a2.base)}]
\dv{1}{(0, 0)}{1}\dv{2}{(1.5, 0)}{2}
\deH{a1}{a2}{4}{2}
\end{tikzpicture} \hfil
&
(0, 0, 1)\newline (1, 2, 0)
\\\hline
\hfil $\alpha_{2}$ \hfil &
\hfil \begin{tikzpicture}[baseline={(a2.base)}]
\dv{1}{(0, 0)}{1}\dv{2}{(1.5, 0)}{2}
\deH{a1}{a2}{10}{10}
\end{tikzpicture} \hfil
&
(2, 3, 0)\newline (2, 9, 6)
\\\hline
\hfil $\alpha_{3}$ \hfil &
\hfil \begin{tikzpicture}[baseline={(a2.base)}]
\dv{1}{(0, 0)}{1}\dv{2}{(1.5, 0)}{2}
\deH{a1}{a2}{3}{2}
\end{tikzpicture} \hfil
&
(1, 0, 0)\newline (0, 2, 1)
\end{tabular}\\\hline
8 &
\begin{tikzpicture}[baseline={(a2.base)}]
\dv{1}{(-1.25, 0)}{1}\dv{2}{(0, 0)}{2}\dv{3}{(1.25, 0)}{3}
\deGtwo{a1}{a2}\deCtwo{a3}{a2}
\end{tikzpicture}
&\renewcommand{\arraystretch}{1.75}%
\begin{tabular}{m{1cm}|m{3cm}|m{4cm}}
\hfil $\alpha_{1}$ \hfil &
\hfil \begin{tikzpicture}[baseline={(a2.base)}]
\dv{1}{(0, 0)}{1}\dv{2}{(1.5, 0)}{2}
\deH{a1}{a2}{2}{4}
\end{tikzpicture} \hfil
&
(0, 0, 1)\newline (3, 2, 0)
\\\hline
\hfil $\alpha_{2}$ \hfil &
\hfil \begin{tikzpicture}[baseline={(a2.base)}]
\dv{1}{(0, 0)}{1}\dv{2}{(1.5, 0)}{2}
\deH{a1}{a2}{2}{6}
\end{tikzpicture} \hfil
&
(2, 1, 0)\newline (0, 1, 2)
\\\hline
\hfil $\alpha_{3}$ \hfil &
\hfil \begin{tikzpicture}[baseline={(a2.base)}]
\dv{1}{(0, 0)}{1}\dv{2}{(1.5, 0)}{2}
\deH{a1}{a2}{4}{4}
\end{tikzpicture} \hfil
&
(1, 0, 0)\newline (1, 2, 2)
\end{tabular}\\\hline
9 &
\begin{tikzpicture}[baseline={(a2.base)}]
\dv{1}{(-1.25, 0)}{1}\dv{2}{(0, 0)}{2}\dv{3}{(1.25, 0)}{3}
\deGtwo{a2}{a1}\deGtwo{a3}{a2}
\end{tikzpicture}
&\renewcommand{\arraystretch}{1.75}%
\begin{tabular}{m{1cm}|m{3cm}|m{4cm}}
\hfil $\alpha_{1}$ \hfil &
\hfil \begin{tikzpicture}[baseline={(a2.base)}]
\dv{1}{(0, 0)}{1}\dv{2}{(1.5, 0)}{2}
\deH{a1}{a2}{2}{6}
\end{tikzpicture} \hfil
&
(0, 0, 1)\newline (1, 2, 0)
\\\hline
\hfil $\alpha_{2}$ \hfil &
\hfil \begin{tikzpicture}[baseline={(a2.base)}]
\dv{1}{(0, 0)}{1}\dv{2}{(1.5, 0)}{2}
\deH{a1}{a2}{2}{18}
\end{tikzpicture} \hfil
&
(0, 1, 2)\newline (2, 3, 0)
\\\hline
\hfil $\alpha_{3}$ \hfil &
\hfil \begin{tikzpicture}[baseline={(a2.base)}]
\dv{1}{(0, 0)}{1}\dv{2}{(1.5, 0)}{2}
\deH{a1}{a2}{6}{2}
\end{tikzpicture} \hfil
&
(1, 0, 0)\newline (0, 2, 3)
\end{tabular}\\\hline
10 &
\begin{tikzpicture}[baseline={(a2.base)}]
\dv{1}{(-1.25, 0)}{1}\dv{2}{(0, 0)}{2}\dv{3}{(1.25, 0)}{3}
\deGtwo{a2}{a1}\deGtwo{a2}{a3}
\end{tikzpicture}
&\renewcommand{\arraystretch}{1.75}%
\begin{tabular}{m{1cm}|m{3cm}|m{4cm}}
\hfil $\alpha_{1}$ \hfil &
\hfil \begin{tikzpicture}[baseline={(a2.base)}]
\dv{1}{(0, 0)}{1}\dv{2}{(1.5, 0)}{2}
\deH{a1}{a2}{6}{2}
\end{tikzpicture} \hfil
&
(0, 0, 1)\newline (1, 2, 0)
\\\hline
\hfil $\alpha_{2}$ \hfil &
\hfil \begin{tikzpicture}[baseline={(a2.base)}]
\dv{1}{(0, 0)}{1}\dv{2}{(1.5, 0)}{2}
\deH{a1}{a2}{6}{6}
\end{tikzpicture} \hfil
&
(2, 3, 0)\newline (0, 3, 2)
\end{tabular}\\\hline
11 &
\begin{tikzpicture}[baseline={(a2.base)}]
\dv{1}{(-1.25, 0)}{1}\dv{2}{(0, 0)}{2}\dv{3}{(1.25, 0)}{3}
\deGtwo{a1}{a2}\deGtwo{a3}{a2}
\end{tikzpicture}
&\renewcommand{\arraystretch}{1.75}%
\begin{tabular}{m{1cm}|m{3cm}|m{4cm}}
\hfil $\alpha_{1}$ \hfil &
\hfil \begin{tikzpicture}[baseline={(a2.base)}]
\dv{1}{(0, 0)}{1}\dv{2}{(1.5, 0)}{2}
\deH{a1}{a2}{2}{6}
\end{tikzpicture} \hfil
&
(0, 0, 1)\newline (3, 2, 0)
\\\hline
\hfil $\alpha_{2}$ \hfil &
\hfil \begin{tikzpicture}[baseline={(a2.base)}]
\dv{1}{(0, 0)}{1}\dv{2}{(1.5, 0)}{2}
\deH{a1}{a2}{6}{6}
\end{tikzpicture} \hfil
&
(2, 1, 0)\newline (0, 1, 2)
\end{tabular}\\\hline
12 &
\begin{tikzpicture}[baseline={(a2.base)}]
\dv{1}{(0,0)}{1}\dv{2}{(-30:1)}{2}\dv{3}{(-90:1)}{3}
\deA{a1}{a2}\deAoo{a1}{a3}\deA{a2}{a3}
\end{tikzpicture}
&
\begin{tabular}{m{1cm}|m{3cm}|m{4cm}}
\hfil $\alpha_{2}$ \hfil &
\hfil {\Huge $\varnothing$} \rule[-6pt]{0pt}{28pt} \hfil
&
\rule{0pt}{5.1em}
\end{tabular}\\\hline
13 &
\begin{tikzpicture}[baseline={(a2.base)}]
\dv{1}{(0,0)}{1}\dv{2}{(-30:1)}{2}\dv{3}{(-90:1)}{3}
\deA{a1}{a2}\deAoo{a1}{a3}\deAoo{a2}{a3}
\end{tikzpicture}
&\renewcommand{\arraystretch}{1.75}%
\begin{tabular}{m{1cm}|m{3cm}|m{4cm}}
\hfil $\alpha_{1}$ \hfil &
\hfil {\Huge $\varnothing$} \rule[-6pt]{0pt}{28pt} \hfil
&

\\\hline
\hfil $\alpha_{3}$ \hfil &
\hfil {\Huge $\varnothing$} \rule[-6pt]{0pt}{28pt} \hfil
&

\end{tabular}\\\hline
14 &
\begin{tikzpicture}[baseline={(a2.base)}]
\dv{1}{(0,0)}{1}\dv{2}{(-30:1)}{2}\dv{3}{(-90:1)}{3}
\deA{a1}{a2}\deAtt{a1}{a3}\deAtt{a2}{a3}
\end{tikzpicture}
&\renewcommand{\arraystretch}{1.75}%
\begin{tabular}{m{1cm}|m{3cm}|m{4cm}}
\hfil $\alpha_{1}$ \hfil &
\hfil {\Huge $\varnothing$} \rule[-6pt]{0pt}{28pt} \hfil
&

\\\hline
\hfil $\alpha_{3}$ \hfil &
\hfil {\Huge $\varnothing$} \rule[-6pt]{0pt}{28pt} \hfil
&

\end{tabular}\\\hline
15 &
\begin{tikzpicture}[baseline={(a2.base)}]
\dv{1}{(0,0)}{1}\dv{2}{(-30:1)}{2}\dv{3}{(-90:1)}{3}
\deA{a1}{a2}\deAtt{a3}{a1}\deAtt{a3}{a2}
\end{tikzpicture}
&\renewcommand{\arraystretch}{1.75}%
\begin{tabular}{m{1cm}|m{3cm}|m{4cm}}
\hfil $\alpha_{1}$ \hfil &
\hfil {\Huge $\varnothing$} \rule[-6pt]{0pt}{28pt} \hfil
&

\\\hline
\hfil $\alpha_{3}$ \hfil &
\hfil {\Huge $\varnothing$} \rule[-6pt]{0pt}{28pt} \hfil
&

\end{tabular}\\\hline
16 &
\begin{tikzpicture}[baseline={(a2.base)}]
\dv{1}{(0,0)}{1}\dv{2}{(-30:1)}{2}\dv{3}{(-90:1)}{3}
\deCtwo{a1}{a2}\deCtwo{a3}{a1}\deAtt{a3}{a2}
\end{tikzpicture}
&\renewcommand{\arraystretch}{1.75}%
\begin{tabular}{m{1cm}|m{3cm}|m{4cm}}
\hfil $\alpha_{1}$ \hfil &
\hfil \begin{tikzpicture}[baseline={(a2.base)}]
\dv{1}{(0, 0)}{1}\dv{2}{(1.5, 0)}{2}
\deH{a1}{a2}{34}{34}
\end{tikzpicture} \hfil
&
(1, 0, 2)\newline (7, 6, 2)
\\\hline
\hfil $\alpha_{2}$ \hfil &
\hfil \begin{tikzpicture}[baseline={(a1.base)}]
\dv{1}{(0, 0)}{1}

\end{tikzpicture} \hfil
&
(2, 1, 0)
\\\hline
\hfil $\alpha_{3}$ \hfil &
\hfil {\Huge $\varnothing$} \rule[-6pt]{0pt}{28pt} \hfil
&

\end{tabular}\\\hline
17 &
\begin{tikzpicture}[baseline={(a2.base)}]
\dv{1}{(0,0)}{1}\dv{2}{(-30:1)}{2}\dv{3}{(-90:1)}{3}
\deCtwo{a2}{a1}\deAoo{a1}{a3}\deCtwo{a2}{a3}
\end{tikzpicture}
&\renewcommand{\arraystretch}{1.75}%
\begin{tabular}{m{1cm}|m{3cm}|m{4cm}}
\hfil $\alpha_{1}$ \hfil &
\hfil \begin{tikzpicture}[baseline={(a1.base)}]
\dv{1}{(0, 0)}{1}

\end{tikzpicture} \hfil
&
(1, 2, 0)
\\\hline
\hfil $\alpha_{2}$ \hfil &
\hfil {\Huge $\varnothing$} \rule[-6pt]{0pt}{28pt} \hfil
&

\end{tabular}\\\hline
18 &
\begin{tikzpicture}[baseline={(a2.base)}]
\dv{1}{(0,0)}{1}\dv{2}{(-30:1)}{2}\dv{3}{(-90:1)}{3}
\deCtwo{a1}{a2}\deAoo{a1}{a3}\deCtwo{a3}{a2}
\end{tikzpicture}
&\renewcommand{\arraystretch}{1.75}%
\begin{tabular}{m{1cm}|m{3cm}|m{4cm}}
\hfil $\alpha_{1}$ \hfil &
\hfil {\Huge $\varnothing$} \rule[-6pt]{0pt}{28pt} \hfil
&

\\\hline
\hfil $\alpha_{2}$ \hfil &
\hfil \begin{tikzpicture}[baseline={(a2.base)}]
\dv{1}{(0, 0)}{1}\dv{2}{(1.5, 0)}{2}
\deH{a1}{a2}{6}{6}
\end{tikzpicture} \hfil
&
(2, 1, 0)\newline (0, 1, 2)
\end{tabular}\\\hline
19 &
\begin{tikzpicture}[baseline={(a2.base)}]
\dv{1}{(0,0)}{1}\dv{2}{(-30:1)}{2}\dv{3}{(-90:1)}{3}
\deAoo{a1}{a2}\deAoo{a1}{a3}\deAoo{a2}{a3}
\end{tikzpicture}
&
\begin{tabular}{m{1cm}|m{3cm}|m{4cm}}
\hfil $\alpha_{1}$ \hfil &
\hfil {\Huge $\varnothing$} \rule[-6pt]{0pt}{28pt} \hfil
&
\rule{0pt}{5.1em}
\end{tabular}\\\hline
20 &
\begin{tikzpicture}[baseline={(a2.base)}]
\dv{1}{(0,0)}{1}\dv{2}{(-30:1)}{2}\dv{3}{(-90:1)}{3}
\deAoo{a1}{a2}\deGtwo{a3}{a1}\deGtwo{a3}{a2}
\end{tikzpicture}
&\renewcommand{\arraystretch}{1.75}%
\begin{tabular}{m{1cm}|m{3cm}|m{4cm}}
\hfil $\alpha_{1}$ \hfil &
\hfil \begin{tikzpicture}[baseline={(a1.base)}]
\dv{1}{(0, 0)}{1}

\end{tikzpicture} \hfil
&
(1, 0, 2)
\\\hline
\hfil $\alpha_{3}$ \hfil &
\hfil \begin{tikzpicture}[baseline={(a2.base)}]
\dv{1}{(0, 0)}{1}\dv{2}{(1.5, 0)}{2}
\deH{a1}{a2}{14}{14}
\end{tikzpicture} \hfil
&
(2, 0, 3)\newline (0, 2, 3)
\end{tabular}\\\hline
21 &
\begin{tikzpicture}[baseline={(a2.base)}]
\dv{1}{(0,0)}{1}\dv{2}{(-30:1)}{2}\dv{3}{(-90:1)}{3}
\deAoo{a1}{a2}\deGtwo{a1}{a3}\deGtwo{a2}{a3}
\end{tikzpicture}
&\renewcommand{\arraystretch}{1.75}%
\begin{tabular}{m{1cm}|m{3cm}|m{4cm}}
\hfil $\alpha_{1}$ \hfil &
\hfil \begin{tikzpicture}[baseline={(a1.base)}]
\dv{1}{(0, 0)}{1}

\end{tikzpicture} \hfil
&
(3, 0, 2)
\\\hline
\hfil $\alpha_{3}$ \hfil &
\hfil \begin{tikzpicture}[baseline={(a2.base)}]
\dv{1}{(0, 0)}{1}\dv{2}{(1.5, 0)}{2}
\deH{a1}{a2}{14}{14}
\end{tikzpicture} \hfil
&
(2, 0, 1)\newline (0, 2, 1)
\end{tabular}\\\hline
22 &
\begin{tikzpicture}[baseline={(a2.base)}]
\dv{1}{(0,0)}{1}\dv{2}{(-30:1)}{2}\dv{3}{(-90:1)}{3}
\deAoo{a1}{a2}\deAtt{a3}{a1}\deAtt{a3}{a2}
\end{tikzpicture}
&\renewcommand{\arraystretch}{1.75}%
\begin{tabular}{m{1cm}|m{3cm}|m{4cm}}
\hfil $\alpha_{1}$ \hfil &
\hfil {\Huge $\varnothing$} \rule[-6pt]{0pt}{28pt} \hfil
&

\\\hline
\hfil $\alpha_{3}$ \hfil &
\hfil {\Huge $\varnothing$} \rule[-6pt]{0pt}{28pt} \hfil
&

\end{tabular}\\\hline
23 &
\begin{tikzpicture}[baseline={(a2.base)}]
\dv{1}{(0,0)}{1}\dv{2}{(-30:1)}{2}\dv{3}{(-90:1)}{3}
\deAoo{a1}{a2}\deAtt{a1}{a3}\deAtt{a2}{a3}
\end{tikzpicture}
&\renewcommand{\arraystretch}{1.75}%
\begin{tabular}{m{1cm}|m{3cm}|m{4cm}}
\hfil $\alpha_{1}$ \hfil &
\hfil {\Huge $\varnothing$} \rule[-6pt]{0pt}{28pt} \hfil
&

\\\hline
\hfil $\alpha_{3}$ \hfil &
\hfil {\Huge $\varnothing$} \rule[-6pt]{0pt}{28pt} \hfil
&

\end{tabular}\\\hline
24 &
\begin{tikzpicture}[baseline={(a2.base)}]
\dv{1}{(-1.25, 0)}{1}\dv{2}{(0, 0)}{2}\dv{3}{(1.25, 0)}{3}
\deAoo{a1}{a2}\deA{a2}{a3}
\end{tikzpicture}
&\renewcommand{\arraystretch}{1.75}%
\begin{tabular}{m{1cm}|m{3cm}|m{4cm}}
\hfil $\alpha_{1}$ \hfil &
\hfil \begin{tikzpicture}[baseline={(a1.base)}]
\dv{1}{(0, 0)}{1}

\end{tikzpicture} \hfil
&
(0, 0, 1)
\\\hline
\hfil $\alpha_{3}$ \hfil &
\hfil \begin{tikzpicture}[baseline={(a1.base)}]
\dv{1}{(0, 0)}{1}

\end{tikzpicture} \hfil
&
(1, 0, 0)
\end{tabular}\\\hline
25 &
\begin{tikzpicture}[baseline={(a2.base)}]
\dv{1}{(-1.25, 0)}{1}\dv{2}{(0, 0)}{2}\dv{3}{(1.25, 0)}{3}
\deAtt{a1}{a2}\deA{a2}{a3}
\end{tikzpicture}
&\renewcommand{\arraystretch}{1.75}%
\begin{tabular}{m{1cm}|m{3cm}|m{4cm}}
\hfil $\alpha_{1}$ \hfil &
\hfil \begin{tikzpicture}[baseline={(a1.base)}]
\dv{1}{(0, 0)}{1}

\end{tikzpicture} \hfil
&
(0, 0, 1)
\\\hline
\hfil $\alpha_{3}$ \hfil &
\hfil \begin{tikzpicture}[baseline={(a1.base)}]
\dv{1}{(0, 0)}{1}

\end{tikzpicture} \hfil
&
(1, 0, 0)
\end{tabular}\\\hline
26 &
\begin{tikzpicture}[baseline={(a2.base)}]
\dv{1}{(-1.25, 0)}{1}\dv{2}{(0, 0)}{2}\dv{3}{(1.25, 0)}{3}
\deAtt{a2}{a1}\deA{a2}{a3}
\end{tikzpicture}
&\renewcommand{\arraystretch}{1.75}%
\begin{tabular}{m{1cm}|m{3cm}|m{4cm}}
\hfil $\alpha_{1}$ \hfil &
\hfil \begin{tikzpicture}[baseline={(a1.base)}]
\dv{1}{(0, 0)}{1}

\end{tikzpicture} \hfil
&
(0, 0, 1)
\\\hline
\hfil $\alpha_{3}$ \hfil &
\hfil \begin{tikzpicture}[baseline={(a1.base)}]
\dv{1}{(0, 0)}{1}

\end{tikzpicture} \hfil
&
(1, 0, 0)
\end{tabular}\\\hline
27 &
\begin{tikzpicture}[baseline={(a2.base)}]
\dv{1}{(-1.25, 0)}{1}\dv{2}{(0, 0)}{2}\dv{3}{(1.25, 0)}{3}
\deAoo{a1}{a2}\deCtwo{a2}{a3}
\end{tikzpicture}
&\renewcommand{\arraystretch}{1.75}%
\begin{tabular}{m{1cm}|m{3cm}|m{4cm}}
\hfil $\alpha_{1}$ \hfil &
\hfil \begin{tikzpicture}[baseline={(a1.base)}]
\dv{1}{(0, 0)}{1}

\end{tikzpicture} \hfil
&
(0, 0, 1)
\\\hline
\hfil $\alpha_{2}$ \hfil &
\hfil {\Huge $\varnothing$} \rule[-6pt]{0pt}{28pt} \hfil
&

\\\hline
\hfil $\alpha_{3}$ \hfil &
\hfil \begin{tikzpicture}[baseline={(a2.base)}]
\dv{1}{(0, 0)}{1}\dv{2}{(1.5, 0)}{2}
\deH{a1}{a2}{2}{4}
\end{tikzpicture} \hfil
&
(1, 0, 0)\newline (0, 2, 1)
\end{tabular}\\\hline
28 &
\begin{tikzpicture}[baseline={(a2.base)}]
\dv{1}{(-1.25, 0)}{1}\dv{2}{(0, 0)}{2}\dv{3}{(1.25, 0)}{3}
\deAoo{a1}{a2}\deCtwo{a3}{a2}
\end{tikzpicture}
&\renewcommand{\arraystretch}{1.75}%
\begin{tabular}{m{1cm}|m{3cm}|m{4cm}}
\hfil $\alpha_{1}$ \hfil &
\hfil \begin{tikzpicture}[baseline={(a1.base)}]
\dv{1}{(0, 0)}{1}

\end{tikzpicture} \hfil
&
(0, 0, 1)
\\\hline
\hfil $\alpha_{2}$ \hfil &
\hfil \begin{tikzpicture}[baseline={(a1.base)}]
\dv{1}{(0, 0)}{1}

\end{tikzpicture} \hfil
&
(0, 1, 2)
\\\hline
\hfil $\alpha_{3}$ \hfil &
\hfil \begin{tikzpicture}[baseline={(a2.base)}]
\dv{1}{(0, 0)}{1}\dv{2}{(1.5, 0)}{2}
\deH{a1}{a2}{6}{6}
\end{tikzpicture} \hfil
&
(1, 0, 0)\newline (1, 4, 4)
\end{tabular}\\\hline
29 &
\begin{tikzpicture}[baseline={(a2.base)}]
\dv{1}{(-1.25, 0)}{1}\dv{2}{(0, 0)}{2}\dv{3}{(1.25, 0)}{3}
\deCtwo{a2}{a1}\deAtt{a3}{a2}
\end{tikzpicture}
&\renewcommand{\arraystretch}{1.75}%
\begin{tabular}{m{1cm}|m{3cm}|m{4cm}}
\hfil $\alpha_{1}$ \hfil &
\hfil \begin{tikzpicture}[baseline={(a2.base)}]
\dv{1}{(0, 0)}{1}\dv{2}{(1.5, 0)}{2}
\deH{a1}{a2}{1}{8}
\end{tikzpicture} \hfil
&
(0, 0, 1)\newline (1, 2, 0)
\\\hline
\hfil $\alpha_{2}$ \hfil &
\hfil {\Huge $\varnothing$} \rule[-6pt]{0pt}{28pt} \hfil
&

\\\hline
\hfil $\alpha_{3}$ \hfil &
\hfil \begin{tikzpicture}[baseline={(a1.base)}]
\dv{1}{(0, 0)}{1}

\end{tikzpicture} \hfil
&
(1, 0, 0)
\end{tabular}\\\hline
30 &
\begin{tikzpicture}[baseline={(a2.base)}]
\dv{1}{(-1.25, 0)}{1}\dv{2}{(0, 0)}{2}\dv{3}{(1.25, 0)}{3}
\deCtwo{a1}{a2}\deAtt{a2}{a3}
\end{tikzpicture}
&\renewcommand{\arraystretch}{1.75}%
\begin{tabular}{m{1cm}|m{3cm}|m{4cm}}
\hfil $\alpha_{1}$ \hfil &
\hfil \begin{tikzpicture}[baseline={(a2.base)}]
\dv{1}{(0, 0)}{1}\dv{2}{(1.5, 0)}{2}
\deH{a1}{a2}{6}{6}
\end{tikzpicture} \hfil
&
(0, 0, 1)\newline (8, 8, 1)
\\\hline
\hfil $\alpha_{2}$ \hfil &
\hfil \begin{tikzpicture}[baseline={(a1.base)}]
\dv{1}{(0, 0)}{1}

\end{tikzpicture} \hfil
&
(2, 1, 0)
\\\hline
\hfil $\alpha_{3}$ \hfil &
\hfil \begin{tikzpicture}[baseline={(a1.base)}]
\dv{1}{(0, 0)}{1}

\end{tikzpicture} \hfil
&
(1, 0, 0)
\end{tabular}\\\hline
31 &
\begin{tikzpicture}[baseline={(a2.base)}]
\dv{1}{(-1.25, 0)}{1}\dv{2}{(0, 0)}{2}\dv{3}{(1.25, 0)}{3}
\deAtt{a2}{a1}\deCtwo{a2}{a3}
\end{tikzpicture}
&\renewcommand{\arraystretch}{1.75}%
\begin{tabular}{m{1cm}|m{3cm}|m{4cm}}
\hfil $\alpha_{1}$ \hfil &
\hfil \begin{tikzpicture}[baseline={(a1.base)}]
\dv{1}{(0, 0)}{1}

\end{tikzpicture} \hfil
&
(0, 0, 1)
\\\hline
\hfil $\alpha_{2}$ \hfil &
\hfil {\Huge $\varnothing$} \rule[-6pt]{0pt}{28pt} \hfil
&

\\\hline
\hfil $\alpha_{3}$ \hfil &
\hfil \begin{tikzpicture}[baseline={(a2.base)}]
\dv{1}{(0, 0)}{1}\dv{2}{(1.5, 0)}{2}
\deH{a1}{a2}{4}{2}
\end{tikzpicture} \hfil
&
(1, 0, 0)\newline (0, 2, 1)
\end{tabular}\\\hline
32 &
\begin{tikzpicture}[baseline={(a2.base)}]
\dv{1}{(-1.25, 0)}{1}\dv{2}{(0, 0)}{2}\dv{3}{(1.25, 0)}{3}
\deAtt{a1}{a2}\deCtwo{a3}{a2}
\end{tikzpicture}
&\renewcommand{\arraystretch}{1.75}%
\begin{tabular}{m{1cm}|m{3cm}|m{4cm}}
\hfil $\alpha_{1}$ \hfil &
\hfil \begin{tikzpicture}[baseline={(a1.base)}]
\dv{1}{(0, 0)}{1}

\end{tikzpicture} \hfil
&
(0, 0, 1)
\\\hline
\hfil $\alpha_{2}$ \hfil &
\hfil \begin{tikzpicture}[baseline={(a1.base)}]
\dv{1}{(0, 0)}{1}

\end{tikzpicture} \hfil
&
(0, 1, 2)
\\\hline
\hfil $\alpha_{3}$ \hfil &
\hfil \begin{tikzpicture}[baseline={(a2.base)}]
\dv{1}{(0, 0)}{1}\dv{2}{(1.5, 0)}{2}
\deH{a1}{a2}{6}{6}
\end{tikzpicture} \hfil
&
(1, 0, 0)\newline (1, 2, 2)
\end{tabular}\\\hline
33 &
\begin{tikzpicture}[baseline={(a2.base)}]
\dv{1}{(-1.25, 0)}{1}\dv{2}{(0, 0)}{2}\dv{3}{(1.25, 0)}{3}
\deAoo{a1}{a2}\deAoo{a2}{a3}
\end{tikzpicture}
&\renewcommand{\arraystretch}{1.75}%
\begin{tabular}{m{1cm}|m{3cm}|m{4cm}}
\hfil $\alpha_{1}$ \hfil &
\hfil \begin{tikzpicture}[baseline={(a1.base)}]
\dv{1}{(0, 0)}{1}

\end{tikzpicture} \hfil
&
(0, 0, 1)
\\\hline
\hfil $\alpha_{2}$ \hfil &
\hfil {\Huge $\varnothing$} \rule[-6pt]{0pt}{28pt} \hfil
&

\end{tabular}\\\hline
34 &
\begin{tikzpicture}[baseline={(a2.base)}]
\dv{1}{(-1.25, 0)}{1}\dv{2}{(0, 0)}{2}\dv{3}{(1.25, 0)}{3}
\deAoo{a1}{a2}\deGtwo{a3}{a2}
\end{tikzpicture}
&\renewcommand{\arraystretch}{1.75}%
\begin{tabular}{m{1cm}|m{3cm}|m{4cm}}
\hfil $\alpha_{1}$ \hfil &
\hfil \begin{tikzpicture}[baseline={(a1.base)}]
\dv{1}{(0, 0)}{1}

\end{tikzpicture} \hfil
&
(0, 0, 1)
\\\hline
\hfil $\alpha_{2}$ \hfil &
\hfil \begin{tikzpicture}[baseline={(a1.base)}]
\dv{1}{(0, 0)}{1}

\end{tikzpicture} \hfil
&
(0, 1, 2)
\\\hline
\hfil $\alpha_{3}$ \hfil &
\hfil \begin{tikzpicture}[baseline={(a2.base)}]
\dv{1}{(0, 0)}{1}\dv{2}{(1.5, 0)}{2}
\deH{a1}{a2}{4}{4}
\end{tikzpicture} \hfil
&
(1, 0, 0)\newline (0, 2, 3)
\end{tabular}\\\hline
35 &
\begin{tikzpicture}[baseline={(a2.base)}]
\dv{1}{(-1.25, 0)}{1}\dv{2}{(0, 0)}{2}\dv{3}{(1.25, 0)}{3}
\deAoo{a1}{a2}\deGtwo{a2}{a3}
\end{tikzpicture}
&\renewcommand{\arraystretch}{1.75}%
\begin{tabular}{m{1cm}|m{3cm}|m{4cm}}
\hfil $\alpha_{1}$ \hfil &
\hfil \begin{tikzpicture}[baseline={(a1.base)}]
\dv{1}{(0, 0)}{1}

\end{tikzpicture} \hfil
&
(0, 0, 1)
\\\hline
\hfil $\alpha_{2}$ \hfil &
\hfil \begin{tikzpicture}[baseline={(a1.base)}]
\dv{1}{(0, 0)}{1}

\end{tikzpicture} \hfil
&
(0, 3, 2)
\\\hline
\hfil $\alpha_{3}$ \hfil &
\hfil \begin{tikzpicture}[baseline={(a2.base)}]
\dv{1}{(0, 0)}{1}\dv{2}{(1.5, 0)}{2}
\deH{a1}{a2}{4}{4}
\end{tikzpicture} \hfil
&
(1, 0, 0)\newline (0, 2, 1)
\end{tabular}\\\hline
36 &
\begin{tikzpicture}[baseline={(a2.base)}]
\dv{1}{(-1.25, 0)}{1}\dv{2}{(0, 0)}{2}\dv{3}{(1.25, 0)}{3}
\deAoo{a1}{a2}\deAtt{a3}{a2}
\end{tikzpicture}
&\renewcommand{\arraystretch}{1.75}%
\begin{tabular}{m{1cm}|m{3cm}|m{4cm}}
\hfil $\alpha_{1}$ \hfil &
\hfil \begin{tikzpicture}[baseline={(a1.base)}]
\dv{1}{(0, 0)}{1}

\end{tikzpicture} \hfil
&
(0, 0, 1)
\\\hline
\hfil $\alpha_{2}$ \hfil &
\hfil {\Huge $\varnothing$} \rule[-6pt]{0pt}{28pt} \hfil
&

\\\hline
\hfil $\alpha_{3}$ \hfil &
\hfil \begin{tikzpicture}[baseline={(a1.base)}]
\dv{1}{(0, 0)}{1}

\end{tikzpicture} \hfil
&
(1, 0, 0)
\end{tabular}\\\hline
37 &
\begin{tikzpicture}[baseline={(a2.base)}]
\dv{1}{(-1.25, 0)}{1}\dv{2}{(0, 0)}{2}\dv{3}{(1.25, 0)}{3}
\deAoo{a1}{a2}\deAtt{a2}{a3}
\end{tikzpicture}
&\renewcommand{\arraystretch}{1.75}%
\begin{tabular}{m{1cm}|m{3cm}|m{4cm}}
\hfil $\alpha_{1}$ \hfil &
\hfil \begin{tikzpicture}[baseline={(a1.base)}]
\dv{1}{(0, 0)}{1}

\end{tikzpicture} \hfil
&
(0, 0, 1)
\\\hline
\hfil $\alpha_{2}$ \hfil &
\hfil {\Huge $\varnothing$} \rule[-6pt]{0pt}{28pt} \hfil
&

\\\hline
\hfil $\alpha_{3}$ \hfil &
\hfil \begin{tikzpicture}[baseline={(a1.base)}]
\dv{1}{(0, 0)}{1}

\end{tikzpicture} \hfil
&
(1, 0, 0)
\end{tabular}\\\hline
38 &
\begin{tikzpicture}[baseline={(a2.base)}]
\dv{1}{(-1.25, 0)}{1}\dv{2}{(0, 0)}{2}\dv{3}{(1.25, 0)}{3}
\deAtt{a1}{a2}\deGtwo{a2}{a3}
\end{tikzpicture}
&\renewcommand{\arraystretch}{1.75}%
\begin{tabular}{m{1cm}|m{3cm}|m{4cm}}
\hfil $\alpha_{1}$ \hfil &
\hfil \begin{tikzpicture}[baseline={(a1.base)}]
\dv{1}{(0, 0)}{1}

\end{tikzpicture} \hfil
&
(0, 0, 1)
\\\hline
\hfil $\alpha_{2}$ \hfil &
\hfil \begin{tikzpicture}[baseline={(a1.base)}]
\dv{1}{(0, 0)}{1}

\end{tikzpicture} \hfil
&
(0, 3, 2)
\\\hline
\hfil $\alpha_{3}$ \hfil &
\hfil \begin{tikzpicture}[baseline={(a2.base)}]
\dv{1}{(0, 0)}{1}\dv{2}{(1.5, 0)}{2}
\deH{a1}{a2}{2}{8}
\end{tikzpicture} \hfil
&
(1, 0, 0)\newline (0, 2, 1)
\end{tabular}\\\hline
39 &
\begin{tikzpicture}[baseline={(a2.base)}]
\dv{1}{(-1.25, 0)}{1}\dv{2}{(0, 0)}{2}\dv{3}{(1.25, 0)}{3}
\deAtt{a2}{a1}\deGtwo{a3}{a2}
\end{tikzpicture}
&\renewcommand{\arraystretch}{1.75}%
\begin{tabular}{m{1cm}|m{3cm}|m{4cm}}
\hfil $\alpha_{1}$ \hfil &
\hfil \begin{tikzpicture}[baseline={(a1.base)}]
\dv{1}{(0, 0)}{1}

\end{tikzpicture} \hfil
&
(0, 0, 1)
\\\hline
\hfil $\alpha_{2}$ \hfil &
\hfil \begin{tikzpicture}[baseline={(a1.base)}]
\dv{1}{(0, 0)}{1}

\end{tikzpicture} \hfil
&
(0, 1, 2)
\\\hline
\hfil $\alpha_{3}$ \hfil &
\hfil \begin{tikzpicture}[baseline={(a2.base)}]
\dv{1}{(0, 0)}{1}\dv{2}{(1.5, 0)}{2}
\deH{a1}{a2}{8}{2}
\end{tikzpicture} \hfil
&
(1, 0, 0)\newline (0, 2, 3)
\end{tabular}\\\hline
40 &
\begin{tikzpicture}[baseline={(a2.base)}]
\dv{1}{(-1.25, 0)}{1}\dv{2}{(0, 0)}{2}\dv{3}{(1.25, 0)}{3}
\deGtwo{a2}{a1}\deAtt{a2}{a3}
\end{tikzpicture}
&\renewcommand{\arraystretch}{1.75}%
\begin{tabular}{m{1cm}|m{3cm}|m{4cm}}
\hfil $\alpha_{1}$ \hfil &
\hfil \begin{tikzpicture}[baseline={(a2.base)}]
\dv{1}{(0, 0)}{1}\dv{2}{(1.5, 0)}{2}
\deH{a1}{a2}{8}{2}
\end{tikzpicture} \hfil
&
(0, 0, 1)\newline (1, 2, 0)
\\\hline
\hfil $\alpha_{2}$ \hfil &
\hfil \begin{tikzpicture}[baseline={(a1.base)}]
\dv{1}{(0, 0)}{1}

\end{tikzpicture} \hfil
&
(2, 3, 0)
\\\hline
\hfil $\alpha_{3}$ \hfil &
\hfil \begin{tikzpicture}[baseline={(a1.base)}]
\dv{1}{(0, 0)}{1}

\end{tikzpicture} \hfil
&
(1, 0, 0)
\end{tabular}\\\hline
41 &
\begin{tikzpicture}[baseline={(a2.base)}]
\dv{1}{(-1.25, 0)}{1}\dv{2}{(0, 0)}{2}\dv{3}{(1.25, 0)}{3}
\deGtwo{a1}{a2}\deAtt{a3}{a2}
\end{tikzpicture}
&\renewcommand{\arraystretch}{1.75}%
\begin{tabular}{m{1cm}|m{3cm}|m{4cm}}
\hfil $\alpha_{1}$ \hfil &
\hfil \begin{tikzpicture}[baseline={(a2.base)}]
\dv{1}{(0, 0)}{1}\dv{2}{(1.5, 0)}{2}
\deH{a1}{a2}{2}{8}
\end{tikzpicture} \hfil
&
(0, 0, 1)\newline (3, 2, 0)
\\\hline
\hfil $\alpha_{2}$ \hfil &
\hfil \begin{tikzpicture}[baseline={(a1.base)}]
\dv{1}{(0, 0)}{1}

\end{tikzpicture} \hfil
&
(2, 1, 0)
\\\hline
\hfil $\alpha_{3}$ \hfil &
\hfil \begin{tikzpicture}[baseline={(a1.base)}]
\dv{1}{(0, 0)}{1}

\end{tikzpicture} \hfil
&
(1, 0, 0)
\end{tabular}\\\hline
42 &
\begin{tikzpicture}[baseline={(a2.base)}]
\dv{1}{(-1.25, 0)}{1}\dv{2}{(0, 0)}{2}\dv{3}{(1.25, 0)}{3}
\deAtt{a2}{a1}\deAtt{a3}{a2}
\end{tikzpicture}
&\renewcommand{\arraystretch}{1.75}%
\begin{tabular}{m{1cm}|m{3cm}|m{4cm}}
\hfil $\alpha_{1}$ \hfil &
\hfil \begin{tikzpicture}[baseline={(a1.base)}]
\dv{1}{(0, 0)}{1}

\end{tikzpicture} \hfil
&
(0, 0, 1)
\\\hline
\hfil $\alpha_{2}$ \hfil &
\hfil {\Huge $\varnothing$} \rule[-6pt]{0pt}{28pt} \hfil
&

\\\hline
\hfil $\alpha_{3}$ \hfil &
\hfil \begin{tikzpicture}[baseline={(a1.base)}]
\dv{1}{(0, 0)}{1}

\end{tikzpicture} \hfil
&
(1, 0, 0)
\end{tabular}\\\hline
43 &
\begin{tikzpicture}[baseline={(a2.base)}]
\dv{1}{(-1.25, 0)}{1}\dv{2}{(0, 0)}{2}\dv{3}{(1.25, 0)}{3}
\deAtt{a2}{a1}\deAtt{a2}{a3}
\end{tikzpicture}
&\renewcommand{\arraystretch}{1.75}%
\begin{tabular}{m{1cm}|m{3cm}|m{4cm}}
\hfil $\alpha_{1}$ \hfil &
\hfil \begin{tikzpicture}[baseline={(a1.base)}]
\dv{1}{(0, 0)}{1}

\end{tikzpicture} \hfil
&
(0, 0, 1)
\\\hline
\hfil $\alpha_{2}$ \hfil &
\hfil {\Huge $\varnothing$} \rule[-6pt]{0pt}{28pt} \hfil
&

\end{tabular}\\\hline
44 &
\begin{tikzpicture}[baseline={(a2.base)}]
\dv{1}{(-1.25, 0)}{1}\dv{2}{(0, 0)}{2}\dv{3}{(1.25, 0)}{3}
\deAtt{a1}{a2}\deAtt{a3}{a2}
\end{tikzpicture}
&\renewcommand{\arraystretch}{1.75}%
\begin{tabular}{m{1cm}|m{3cm}|m{4cm}}
\hfil $\alpha_{1}$ \hfil &
\hfil \begin{tikzpicture}[baseline={(a1.base)}]
\dv{1}{(0, 0)}{1}

\end{tikzpicture} \hfil
&
(0, 0, 1)
\\\hline
\hfil $\alpha_{2}$ \hfil &
\hfil {\Huge $\varnothing$} \rule[-6pt]{0pt}{28pt} \hfil
&

\end{tabular}\\\hline

\end{longtable}}

%% file: table-rk-geq4.tex
{\renewcommand{\arraystretch}{1.5}
\tikzset{every picture/.append style={show background rectangle,inner frame sep=1ex,background rectangle/.style={use as bounding box}}}
\noindent
\begin{longtable}{|l|@{}c@{}|c|@{}c@{}|m{2.85cm}|}
\hline
№ & $\Phi$ & $\alpha$ & $Z_s(\alpha)$ & basis \\\hline

45 &
\begin{tikzpicture}[baseline={(current bounding box.center)}]
\dv{1}{(0,1)}{1}\dv{2}{(1,1)}{2}\dv{3}{(1,0)}{3}\dv{4}{(0,0)}{4}
\deA{a1}{a2}\deA{a1}{a3}\deA{a1}{a4}\deA{a2}{a3}\deA{a2}{a4}\deA{a3}{a4}
\end{tikzpicture}
 & $\alpha_{1}$ &
\raisebox{-4pt}[0pt][0pt]{\Huge $\varnothing$}
& see~\cref{table:k2m24}

\\\hline
46 &
\begin{tikzpicture}[baseline={(current bounding box.center)}]
\dv{1}{(0,0)}{1}\dv{2}{(60:1)}{2}\dv{3}{(1,0)}{3}\dv{4}{(-60:1)}{4}
\deA{a1}{a2}\deA{a1}{a3}\deA{a1}{a4}\deA{a2}{a3}\deA{a3}{a4}
\end{tikzpicture}
 & $\alpha_{4}$ &
{\Huge $\infty$}
&
(0,1,0,0)\newline (8,1,4,6)\newline (4,1,8,6)\newline (28,13,8,18)\newline (8,13,28,18)\newline (28,1,20,24)\newline (20,1,28,24)\newline (52,13,20,36)\newline (20,13,52,36)\newline (92,37,28,60)\newline (28,37,92,60)\newline (104,61,28,66)\newline (28,61,104,66)\newline (188,97,52,120)\newline (52,97,188,120)
\\
\hline

\end{longtable}
}

%% file: kac-moody-x-centralizer.bbl
\newcommand{\etalchar}[1]{$^{#1}$}
\providecommand{\bysame}{\leavevmode\hbox to3em{\hrulefill}\thinspace}
\providecommand{\MR}{\relax\ifhmode\unskip\space\fi MR }
\providecommand{\MRhref}[2]{%
  \href{http://www.ams.org/mathscinet-getitem?mr=#1}{#2}
}
\providecommand{\href}[2]{#2}
\begin{thebibliography}{MRTT53}

\bibitem[BP95]{BilPiaRootStrings}
Y.~Billig and A.~Pianzola, \emph{Root strings with two consecutive real roots},
  Tohoku Math. J.(2) \textbf{47} (1995), no.~3, 391--403.

\bibitem[CCC{\etalchar{+}}10]{CarboneEtAlHyperbolicClassification}
L.~Carbone, S.~Chung, L.~Cobbs, R.~McRae, D.~Nandi, Y.~Naqvi, and D.~Penta,
  \emph{Classification of hyperbolic {D}ynkin diagrams, root lengths and {W}eyl
  group orbits}, J. Phys. A \textbf{43} (2010), no.~15, 155--209.

\bibitem[CKMS15]{CarboneEtAlRank2Hyperbolic}
L.~Carbone, M.~Kownacki, S.~Murray, and S.~Srinivasan, \emph{Root subsystems of
  rank 2 hyperbolic root systems}, arXiv preprint arXiv:1506.05405 (2015).

\bibitem[FP95]{FukudaDoubleDescription}
K.~Fukuda and A.~Prodon, \emph{Double description method revisited},
  Franco-{J}apanese and {F}ranco-{C}hinese {C}onference on {C}ombinatorics and
  {C}omputer {S}cience, Springer, 1995, pp.~91--111.

\bibitem[HMR92]{HurMorRehHomPi0KM}
J.~Hurrelbrink, J.~Morita, and U.~Rehmann, \emph{On the homological $\pi_0$ of
  {K}ac-{M}oody groups over fields}, Cont. Math \textbf{126} (1992), 71--77.

\bibitem[Kac90]{KacInfDimLieAlg}
V.~Kac, \emph{Infinite-{D}imensional {L}ie {A}lgebras}, 3 ed., Cambridge
  University Press, 1990.

\bibitem[MP89]{MooPiaInfRootSystems}
R.~Moody and A.~Pianzola, \emph{On infinite root systems}, Trans. Amer. Math.
  Soc. \textbf{315} (1989), no.~2, 661--696.

\bibitem[MRTT53]{MotzkinDoubleDescription}
T.~Motzkin, H.~Raiffa, G.~Thompson, and R.~Thrall, \emph{The double description
  method}, Contributions to the {T}heory of {G}ames \textbf{2} (1953), no.~28,
  51--73.

\bibitem[PK83]{PetKacInfFlagVarConjTheor}
D.~Peterson and V.~Kac, \emph{Infinite flag varieties and conjugacy theorems},
  Proc. Natl. Acad. Sci. USA \textbf{80} (1983), no.~6, 1778--1782.

\bibitem[Ste71]{SteGenRelCov}
M.~Stein, \emph{Generators, relations and coverings of {C}hevalley groups over
  commutative rings}, Amer. J. Math. \textbf{93} (1971), no.~4, 965--1004.

\end{thebibliography}
